\documentclass[12pt]{article}
\usepackage{amsthm,amsfonts,amssymb,amsmath}
\usepackage{stmaryrd}

\usepackage{cite,hyperref}
\usepackage{epsfig}
\usepackage{url}
\usepackage{xcolor,tikz}
\usetikzlibrary{matrix}
\usepackage{multicol,graphicx}
\usepackage{authblk,fullpage}
\usepackage{bbm}

\usepackage[shortlabels]{enumitem}

\numberwithin{equation}{section}

\setlist{leftmargin=3\parindent,labelindent=3\parindent}
\setlist[enumerate]{%
  leftmargin=3\parindent,%
  align=left,%
  labelwidth=3\parindent,%
  labelsep=0pt%
}
\setlist[enumerate,1]{% 
  label={\normalfont (\thesection.\arabic{equation})}, ref={\normalfont \thesection.\arabic{equation}},
  resume%
}

\newtheorem{thm}[equation]{Theorem}
\newtheorem{cor}[equation]{Corollary}
\newtheorem{lem}[equation]{Lemma}
\newtheorem{prop}[equation]{Proposition}
\newtheorem{claim}[equation]{Claim}

\newtheorem{ques}[equation]{Question}

\theoremstyle{definition}
\newtheorem{defn}[equation]{Definition}
\newtheorem{ass}[equation]{Assumption}

\newtheorem{rem}[equation]{Remark}
\newtheorem{obs}[equation]{Observation}

\newtheorem*{ack}{Acknowledgements}

\theoremstyle{remark}
\newtheorem{case}{Case}
\newtheorem{casee}{Case}

\title{Tur\'an Colourings in Off-Diagonal Ramsey Multiplicity}
\author{Joseph Hyde}
\author{Jae-baek Lee}
\author{Jonathan A. Noel\thanks{Research supported by NSERC Discovery Grant RGPIN-2021-02460, NSERC Early Career Supplement DGECR-2021-00024 and a Start-Up Grant from the University of Victoria.}} 

\affil{\normalsize{Department of Mathematics and Statistics, University of Victoria, Victoria, B.C., Canada.}}
\affil{\texttt{\{josephhyde,dlwoqor0923,noelj\}@uvic.ca}}

\DeclareTextCompositeCommand{\v}{OT1}{l}{l\nobreak\hspace{-.1em}'}
\DeclareTextCompositeCommand{\v}{OT1}{t}{t\nobreak\hspace{-.1em}'\nobreak\hspace{-.15em}}

\DeclareMathOperator{\inj}{inj}

\DeclareMathOperator{\crit}{crit}

\pgfdeclarelayer{edgelayer}
\pgfdeclarelayer{nodelayer}
\pgfsetlayers{edgelayer,nodelayer,main}

\tikzstyle{none}=[inner sep=0pt]
\definecolor{hexcolor0xf81e1c}{rgb}{0.973,0.118,0.110}
\definecolor{hexcolor0x3c00ff}{rgb}{0.235,0.000,1.000}

\tikzstyle{root}=[rectangle, fill=white,draw=black, scale=0.70]
\tikzstyle{vertex}=[circle, fill=white,draw=black, scale=0.55]
\tikzstyle{root_small}=[rectangle, fill=white,draw=black, scale=0.350]
\tikzstyle{vertex_small}=[circle, fill=white,draw=black, scale=0.35]
\tikzstyle{dashededge}=[draw=black, densely dotted]
\tikzstyle{edge}=[draw=black]

\begin{document}

\maketitle
\setcounter{MaxMatrixCols}{16}

\begin{abstract}
The \emph{Ramsey multiplicity constant} of a graph $H$ is the limit as $n$ tends to infinity of the minimum density of monochromatic labeled copies of $H$ in a $2$-edge colouring of $K_n$. Fox and Wigderson recently identified a large family of graphs whose Ramsey multiplicity constants are attained by sequences of ``Tur\'an colourings;'' i.e. colourings in which one of the colour classes forms the edge set of a balanced complete multipartite graph. Each graph in their family comes from taking a connected non-3-colourable graph with a critical edge and adding many pendant edges. We extend their result to an off-diagonal variant of the Ramsey multiplicity constant which involves minimizing a weighted sum of red copies of one graph and blue copies of another. 
\end{abstract}

\section{Introduction}

The central question in the area of ``Ramsey multiplicity'' is: how should one colour the edges of the clique $K_n$, for large $n$, with red and blue to minimize the number of monochromatic labeled copies of a fixed graph $H$? As an ``off-diagonal'' generalization, one could instead minimize a ``suitable linear combination'' of the number of red copies of one graph $H_1$ and blue copies of another graph $H_2$; the coefficients of this linear combination will be specified in Section~\ref{sec:prelims} but, for now, it suffices to think of them as arbitrary positive reals that may depend on $n$. Ramsey multiplicity problems have been extensively studied; see, for example,~\cite{ParczykPokuttaSpiegelSzabo22,Fox08,MossNoel23++,BurrRosta80,FoxWigderson23,JaggerStovicekThomason96,Thomason97,Conlon12,LeeNoel23+,Cummings+13}.

One of the first strategies that comes to mind is to consider a uniformly random colouring. In the diagonal setting, i.e. when $H_1=H_2=H$ for some graph $H$, a random colouring has approximately $(1/2)^{e(H)-1}n^{v(H)}$ monochromatic copies of $H$ with high probability, where $v(H):= |V(H)|$ and $e(H):= |E(H)|$. A graph $H$ is said to be \emph{common} if every colouring has at least this many copies, up to a $(1+o(1))$ factor where the $o(1)$ term tends to $0$ as $n\to\infty$. The notion of common graphs has its origins in the work of Goodman~\cite{Goodman59} and Erd\H{o}s~\cite{Erdos62} in the 1950s and 60s and has been a popular area of research ever since~\cite{GrzesikLeeLidickyVolec22,Thomason89,FirstPaper,KralVolecWei22+,JaggerStovicekThomason96,Hatami+12,LeeNoel23+,KoLee23,SecondPaper}.

After considering random colourings, perhaps the next most natural strategy is to ``pack in'' as many red edges as possible without creating a red copy of $H_1$. Inspired by the well-studied area of ``Ramsey goodness,''~\cite{Erdos47,GerencserGyarfas67,Burr81,Chvatal77,ChvatalHarary72,BurrErdos83,Haslegrave+23,FoxHeWigderson21+,LinLiu21,LinPeng21} one way to do this is to divide $V(K_n)$ into $\chi(H_1)-1$ classes of cardinality $\left\lfloor n/(\chi(H_1)-1)\right\rfloor$ or $\left\lceil n/(\chi(H_1)-1)\right\rceil$ and colour an edge red if it connects vertices in different classes or blue otherwise. Such a colouring, which is referred to as a \emph{Tur\'an colouring},\footnote{We also use the term \emph{Tur\'an colouring} to refer to a colouring in which there are $\chi(H_2)-1$ classes of almost equal size, edges between the classes are blue and edges within the classes are red.} has no red copy of $H_1$ and has at most
\[(1-o(1))\left(\frac{1}{\chi(H_1)-1}\right)^{v(H_2)-k(H_2)} n^{v(H_2)}\]
blue copies of $H_2$, where $k(H_2)$ denotes the number of connected components of $H_2$. In the diagonal setting, Fox and Wigderson~\cite{FoxWigderson23} recently proved that this strategy is optimal for a fairly large family of graphs (see Theorem~\ref{th:FW} below). Prior to their work, there were no examples of uncommon graphs for which the Ramsey multiplicity problem had been solved. 

Our main result (Theorem~\ref{th:hairy} below) extends this theorem of Fox and Wigderson~\cite{FoxWigderson23} to an off-diagonal setting. Stating our results precisely requires some technical definitions which we will formally provide in Section~\ref{sec:prelims}. For the time being, we informally say that $(H_1,H_2)$ is a \emph{bonbon pair} if, for large enough $n$, the only colourings minimizing the ``suitable linear combination'' of the number of red copies of $H_1$ and blue copies of $H_2$ alluded to in the first paragraph are the Tur\'an colourings. Following~\cite{FoxWigderson23}, if $H$ is a graph such that $(H,H)$ is a bonbon pair, then $H$ is said to be a \emph{bonbon}.

For a graph $F$ and $t\geq0$, a \emph{$t$-hairy $F$} is a graph that is created by adding $t$ edges to $F$, one at a time, such that each added edge has exactly one endpoint in $V(F)$. If $H$ is a $t$-hairy $F$ for some $t$, then we simply say that $H$ is a  \emph{hairy} $F$. An edge $e$ of a graph $F$ is \emph{critical} if $\chi(F-e)<\chi(F)$; i.e. removing the edge $e$ (and neither of its vertices) from $F$ decreases the chromatic number. We state the main result of~\cite{FoxWigderson23} and our off-diagonal generalization of it. 

\begin{thm}[Fox and Wigderson~{\cite[Theorem~1.2]{FoxWigderson23}}]
\label{th:FW}
For any connected non-3-colourable graph $F$ that contains a critical edge, there exists $t_0=t_0(F)$ such that, for any $t\geq t_0$, every $t$-hairy $F$ is a bonbon.
\end{thm}

\begin{thm}
\label{th:hairy}
Let $q\in(0,1]$ and let $F_1$ and $F_2$ be non-bipartite graphs, each of which contains a critical edge, such that $\chi(F_1)+\chi(F_2)\geq7$. Then there exists $t_0=t_0(F_1,F_2,q)$ such that if $H_1$ is a $t_1$-hairy $F_1$ and $H_2$ is a $t_2$-hairy $F_2$ with $t_1,t_2\geq t_0$ and 
\[\min\{v(H_1),v(H_2)\}\geq q\cdot\max\{v(H_1),v(H_2)\},\]
then $(H_1,H_2)$ is a bonbon pair.
\end{thm}

One may wonder whether the presence of critical edges and the dependence of $t_0$ on the parameter $q$ are essential in Theorem~\ref{th:hairy}. The next proposition implies that both conditions are necessary. Let $\crit(F)$ denote the number of critical edges in a graph $F$. An explicit form of the function $g$ in the following proposition will be provided in Section~\ref{sec:neg} (see Theorem~\ref{th:imbalance}). 

\begin{prop}
\label{prop:imbalance}
There exists a function $g:\mathbb{R}^4\to \mathbb{R}$ with the property that, if $H_1$ and $H_2$ are non-empty graphs such that
\[e(H_1)> \crit(H_2)\cdot g(\chi(H_1),\chi(H_2),v(H_2),k(H_2)),\]
then $(H_1,H_2)$ is not a bonbon pair.
\end{prop}

In Section~\ref{sec:prelims}, we provide a formal definition of bonbon pairs and an off-diagonal variant of the Ramsey multiplicity constant. The proof of Theorem~\ref{th:hairy}, which is inspired by the proof of Theorem~\ref{th:FW} in~\cite{FoxWigderson23}, is provided in Sections~\ref{sec:hairy} and~\ref{sec:hairier}. First, in Section~\ref{sec:hairy}, we show that an optimal colouring has the ``approximate'' structure of a  Tur\'an colouring; i.e. the vertices can be partitioned into a small number of classes such that edges within the classes are nearly monochromatic. Then, in Section~\ref{sec:hairier}, we refine the structure of the colouring until it precisely matches that of a Tur\'an colouring. In Section~\ref{sec:neg}, we discuss various constructions of colourings which we use to prove a strong form of Proposition~\ref{prop:imbalance} (Theorem~\ref{th:imbalance}). We conclude the paper in Section~\ref{sec:concl} by proposing several open problems.

\section{Formal Definitions}
\label{sec:prelims}

Given graphs $H$ and $G$, a \emph{homomorphism} from $H$ to $G$ is a function $f:V(H)\to V(G)$ such that adjacent pairs of vertices in $H$ are mapped to adjacent pairs of vertices in $G$ and the \emph{homomorphism density} $t(H,G)$ is the probability that a random function from $V(H)$ to $V(G)$ is a homomorphism. That is, $t(H,G)$ is the number of homomorphisms from $H$ to $G$ divided by $v(G)^{v(H)}$. For graphs $H$ and $G$ with $v(H)\leq v(G)$, the \emph{injective homomorphism density} of $H$ in $G$, denoted $t_{\inj}(H,G)$, is the probability that a random injective function from $V(H)$ to $V(G)$ is a homomorphism. If $H$ is a fixed graph and $v(G)$ is large, then there are only $O(v(G)^{v(H)-1})$ non-injective functions from $V(H)$ to $V(G)$, and so 
\begin{equation}\label{eq:tinjt}t(H,G)=t_{\inj}(H,G)+o(1)\end{equation}
where the $o(1)$ term approaches zero as $v(G)\to\infty$. The following is essentially a rephrasing of~\cite[Definition~1.1]{FoxWigderson23}, except that we generalize it slightly to include disconnected graphs. 

\begin{defn}[Fox and Wigderson~\cite{FoxWigderson23}]
\label{def:bonbon}
A non-empty graph $H$ is said to be a \emph{bonbon} if there exists $n_0=n_0(H)$ such that, if $n\geq n_0$ and $G$ is an $n$-vertex graph such that 
\[t_{\inj}(H,G) + t_{\inj}(H,\overline{G})\]
is minimized over all $n$-vertex graphs, then either $G$ or $\overline{G}$ is a Tur\'an graph with $\chi(H)-1$ parts. 
\end{defn}

Let us now extend this definition to an off-diagonal setting. 

\begin{defn}
A pair $(H_1,H_2)$ of non-empty graphs is a \emph{bonbon pair} if there exists $n_0=n_0(H_1,H_2)$ such that, if $n\geq n_0$ and $G$ is an $n$-vertex graph such that
\[(\chi(H_2)-1)^{v(H_1)-k(H_1)}\cdot t_{\inj}(H_1,G)+(\chi(H_1)-1)^{v(H_2)-k(H_2)}\cdot t_{\inj}(H_2,\overline{G})\]
is minimized over all $n$-vertex graphs, then either $G$ is a Tur\'an graph with $\chi(H_1)-1$ parts or $\overline{G}$ is a Tur\'an graph with $\chi(H_2)-1$ parts. 
\end{defn}

Note that a graph $H$ is a bonbon if and only if $(H,H)$ is a bonbon pair. The inspiration behind the above definition comes from several recent papers focusing on off-diagonal generalizations of basic questions in Ramsey multiplicity. For instance, Parczyk, Pokutta, Spiegel and Szab\'o~\cite{ParczykPokuttaSpiegelSzabo22} proved asymptotic bounds on linear combinations of $t_{\inj}(K_s,G)$ and $t_{\inj}(K_t,\overline{G})$ for small $s$ and $t$ and Behague, Morrison and Noel~\cite{FirstPaper,SecondPaper} extended the notion of common graphs to an off-diagonal setting. Moss and Noel~\cite{MossNoel23++} recently introduced an off-diagonal notion of Ramsey multiplicity for general pairs of graphs. In our proof of Theorem~\ref{th:hairy}, we will need the following notion from~\cite{MossNoel23++}.

\begin{defn}[Moss and Noel~\cite{MossNoel23++}]
\label{def:lambdaRM}
For non-empty graphs $H_1$ and $H_2$ and $\lambda\in [0,2]$, define
\[c_\lambda(H_1,H_2):= \lim_{n\to\infty}\left[\min_{G:v(G)=n}\left(\lambda\cdot t(H_1,G)+(2-\lambda)\cdot t(H_2,\overline{G})\right)\right].\]
\end{defn}

\section{Proof of Theorem~\ref{th:hairy}: Rough Structure}
\label{sec:hairy}

The focus of this section is on obtaining an approximate version of Theorem~\ref{th:hairy} (Lemma~\ref{lem:roughStructure} below) which will be refined in the next section to complete the proof of the theorem. 

\begin{rem}
\label{rem:FW}
Fox and Wigderson~\cite{FoxWigderson23} cleverly avoided using the Graph Removal Lemma in their proof of Theorem~\ref{th:FW}. Doing so added a few steps to their argument, but resulted in much better bounds on $t_0$. To keep our paper to a reasonable length, and to differentiate it from~\cite{FoxWigderson23}, we have chosen to present a shorter argument which uses the Removal Lemma at the expense of having poorer control over $t_0$. We remark that better bounds on our $t_0$ could be obtained by following the proof of~\cite[Theorem~1.2]{FoxWigderson23} more closely.
\end{rem}

Throughout the next two sections, we let $q\in (0,1]$ and let $F_1$ and $F_2$ be non-bipartite graphs, each of which contains a critical edge, such that $\chi(F_1)+\chi(F_2)\geq7$. Define $f:=\max\{v(F_1),v(F_2)\}$ and $\chi:=\max\{\chi(F_1),\chi(F_2)\}$. We let $t_0$ be an integer chosen large with respect to $F_1,F_2$ and $q$, which will be specified later. Actually, $t_0$ is defined in terms of a throng of other parameters $\theta,\varepsilon,\delta,\beta,\xi,\gamma$ and $\tau$, where each parameter depends on $F_1,F_2$ and $q$ and the parameters that come before it in the list. The relationships between $F_1,F_2,q,\theta,\varepsilon,\delta,\beta,\xi,\gamma,\tau$ and $t_0$ will be revealed ``as needed'' throughout this section and the next, and will be summarized in the final proof of Theorem~\ref{th:hairy} at the end of Section~\ref{sec:hairier}.

Let $t_1,t_2\geq t_0$ and let $H_1$ be a $t_1$-hairy $F_1$ and $H_2$ be a $t_2$-hairy $F_2$ satisfying
\begin{equation}\label{eq:relSizes}\min\{v(H_1),v(H_2)\}\geq q\cdot \max\{v(H_1),v(H_2)\}.\end{equation}
Note that $\chi(H_i)=\chi(F_i)$ and $k(H_i)=k(F_i)$ for $i\in \{1,2\}$. For the sake of brevity, let 
\[\rho_1:=(\chi(H_2)-1)^{v(H_1)-k(H_1)}\]
\[\rho_2:=(\chi(H_1)-1)^{v(H_2)-k(H_2)}.\]
Note that, by definition,
\begin{equation}\label{eq:samesame}\rho_1\cdot \left(\frac{1}{\chi(F_2)-1}\right)^{v(H_1)-k(F_1)} = \rho_2\left(\frac{1}{\chi(F_1)-1}\right)^{v(H_2)-k(F_2)}=1.\end{equation}
We may assume that $H_1$ and $H_2$ have no singleton components, since any such components do not affect injective homomorphism densities into large enough graphs, nor do they affect $\rho_1$ or $\rho_2$ (since adding a singleton component to a graph $H$ increases both of $v(H)$ and $k(H)$ by one and does not affect $\chi(H)$). 

Let $n_0$ be a large integer which may depend on $H_1,H_2$ and all of the parameters discussed so far, and assume that $n\geq n_0$. For any graph $G$, define
\[m(H_1,H_2;G):=\rho_1\cdot t_{\inj}(H_1,G) + \rho_2\cdot t_{\inj}(H_2,\overline{G}).\]
Here, the letter $m$ stands for ``monochromatic.'' Let $G_1$ be a graph on $n$ vertices chosen so that $\min_{G:v(G)=n}(m(H_1,H_2;G)) = m(H_1,H_2;G_1)$ and let $G_2:=\overline{G_1}$. Our goal in the proof of Theorem~\ref{th:hairy} is to show that either $G_1$ is a Tur\'an graph with $\chi(F_1)-1$ parts or $G_2$ is a Tur\'an graph with $\chi(F_2)-1$ parts. Since $m(H_1,H_2;G_1)$ is at most the value of $m(H_1,H_2;G)$ when $G$ is an $n$-vertex Tur\'an graph with $\chi(F_1)-1$ parts, we have that
\begin{equation}\label{eq:G1good}m(H_1,H_2;G_1) \leq (1-o(1))\rho_2\left(\frac{1}{\chi(F_1)-1}\right)^{v(H_2)-k(F_2)} = 1-o(1)\end{equation}
where the last equality follows from \eqref{eq:samesame}. Note that $m(H_2,H_1;G_2) = m(H_1,H_2; G_1)$ and so it is also at most $1-o(1)$.

It is useful to classify vertices based on their degrees in $G_1$ and $G_2$. Let $V:=V(G_1)=V(G_2)$. For a graph $G$ with vertex set $V$ and a vertex $v \in V$, the \emph{degree} of $v$ in $G$ is the number of edges of $G$ that are connected to $v$, denoted by $d_G(v)$. For $i\in \{1,2\}$, we let $d_i(v):=d_{G_i}(v)$. When interpreting the next definition, recall that $\xi$ is one of the many parameters that appears throughout this section and the next and will be specified in the final proof of Theorem~\ref{th:hairy}. 

\begin{defn}
\label{def:V1V2}
For $i\in\{1,2\}$, define
\[V_i:=\left\{v\in V: d_i(v)\geq \left(1-\frac{1+2\xi}{\chi(F_i)-1}\right)(n-1)\right\}.\]
Also, let $V_0:=V\setminus (V_1\cup V_2)$ and $V_3=V_1\cap V_2$. 
\end{defn}

We may assume the following, without loss of generality.

\begin{ass}
\label{ass:V1<=V2}
$|V_1|\leq |V_2|$. 
\end{ass}

The focus of the rest of this section is on proving the following lemma which determines the ``rough structure'' of $G_1$ and $G_2$. For any two subsets $S,T\subseteq V$ and a graph $G$ with vertex set $V$, define $e_G(S,T)$ to be the number of ordered pairs $(u,v)\in S\times T$ such that $uv\in E(G)$ and let $e_G(S):=\frac{1}{2}e_G(S,S)$. For any $S,T\subseteq V$ and $i\in \{1,2\}$, we let $e_i(S,T):=e_{G_i}(S,T)$ and $e_i(S):=e_{G_i}(S)$.

\begin{lem}
\label{lem:roughStructure}
There exists a partition $A_1,A_2,\dots,A_{\chi(F_2)-1}$ of $V$ such that
\[\sum_{i=1}^{\chi(F_2)-1}e_2(A_i)\leq \varepsilon n^2.\]
\end{lem}

It is worth noting that the proof of Lemma~\ref{lem:roughStructure} does not require $F_1$ and $F_2$ to have critical edges, nor does it require the inequality \eqref{eq:relSizes}; these conditions come into play when seeking the exact structure of an optimal colouring in the next section.

A high-level overview of the proof of Lemma~\ref{lem:roughStructure} is as follows. We first show that $V_3=\emptyset$ and that $V_0$ and $V_1$ are both quite small; specifically $|V_0|\leq \xi n$ and $|V_1|\leq 25\xi n$. Therefore, most of the vertices reside in $V_2$, and thus have a large degree in $G_2$. If the density of $F_2$ in $G_2$ is sufficiently far from zero, then there must be several copies of $F_2$ whose vertices are contained in $V_2$, and each of these copies can be ``extended'' to a copy of $H_2$ in $G_2$ in many ways due to the high $G_2$-degree of vertices in $V_2$. This would lead to a large density of $H_2$ in $G_2$, which would violate \eqref{eq:G1good}. The Graph Removal Lemma then implies that $G_2$ can be made $F_2$-free by deleting a small proportion of its edges. After deleting these edges, we obtain a graph with close to $\left(1-\frac{1}{\chi(F_2)-1}\right)\binom{n}{2}$ edges which is $F_2$-free. The classical Erd\H{o}s--Simonovits Stability Theorem then states that such a graph must be ``close'' to a complete $(\chi(F_2)-1)$-partite graph, which gives us Lemma~\ref{lem:roughStructure}. The rest of the section is devoted to fleshing out the details of these arguments.

\subsection{Analyzing Degrees}

We show that $V_1$ and $V_2$ have empty intersection. The following assumption is useful for proving this, and will be used again later as well:
\begin{equation}
\label{eq:xi1}
0<\xi< \frac{1}{39}.
\end{equation}

\begin{lem}
\label{lem:V1V2disjoint}
$V_3=V_1\cap V_2=\emptyset$.
\end{lem}

\begin{proof}
Suppose not and let $v\in V_1\cap V_2$. Then, since $\chi(F_1),\chi(F_2)\geq3$ and $\chi(F_1)+\chi(F_2)\geq7$, we have
\[n-1= d_1(v)+d_2(v) \geq \left(1-\frac{1+2\xi}{\chi(F_1)-1}\right)(n-1)+\left(1-\frac{1+2\xi}{\chi(F_2)-1}\right)(n-1)\]
\[\geq\left(1-\frac{1+2\xi}{2}\right)(n-1)+\left(1-\frac{1+2\xi}{3}\right)(n-1) = \left(\frac{7}{6} - \frac{5\xi}{3}\right)(n-1).\]
This implies that $\xi\geq 1/10$; however, this contradicts \eqref{eq:xi1}.
\end{proof}

We obtain a bound on the degrees of vertices in $V_i$ for $i\in \{1,2\}$ via a similar argument.

\begin{obs}
\label{obs:Vihigh}
Let $\{i,j\}=\{1,2\}$. If $v\in V_i$, then $d_i(v)>\left(\frac{5}{4}\cdot \frac{1+\xi}{\chi(F_j)-1}\right)n$.
\end{obs}

\begin{proof}
If not, then, since $v\in V_i$ and $n$ is large,
\[\left(1-\frac{1+3\xi}{\chi(F_i)-1}\right)n\leq \left(1-\frac{1+2\xi}{\chi(F_i)-1}\right)(n-1)\leq d_i(v)\leq \left(\frac{5}{4}\cdot \frac{1+\xi}{\chi(F_j)-1}\right)n.\]
Since $\chi(F_1),\chi(F_2)\geq3$ and $\chi(F_1)+\chi(F_2)\geq7$, this implies that
\[1 - \frac{1+3\xi}{3}\leq  \frac{5}{4}\cdot \frac{1+\xi}{2}.\]
However, this can only hold if $\xi\geq 1/39$ which contradicts \eqref{eq:xi1}. 
\end{proof}

Next, we prove that both $V_0$ and $V_1$ are small. Given a graph $G$ on vertex set $V$ and a set $S\subseteq V$, let $G[S]$ be the subgraph of $G$ \emph{induced} by $S$; i.e. the graph with vertex set $V$ and edge set $\{uv\in E(G): u,v\in S\}$. The next lemma says that, for $i\in \{1,2\}$, there cannot be a fairly sizeable set $S$ such that $t(F_i,G_i[S])$ is bounded away from zero and $d_i(v)$ is relatively large for every $v\in S$. To prove this, we assume that $t_0$ is chosen large enough so that the following holds:
\begin{equation}
\label{eq:t0bound1}
(1+\xi)^{t_0}>3/\tau.
\end{equation}

\begin{lem}
\label{lem:inducedSubgraph}
Let $\{i,j\}=\{1,2\}$. If $S$ is a non-empty subset of $V$ such that
\[d_i(v)\geq \left(\frac{1+\xi}{\chi(F_j)-1}\right)n\]
for all $v\in S$, then $t(F_i,G_i[S])\leq \tau\cdot\left(n/|S|\right)^{v(F_i)}$.
\end{lem}

\begin{proof}
Suppose to the contrary that the hypotheses hold but
\[t(F_i,G_i[S])> \tau\cdot\left(n/|S|\right)^{v(F_i)}.\]
The probability that a uniformly random function $\varphi$ from $V(H_i)$ to $V$ is a homomorphism from $H_i$ to $G_i$ is at least the probability that the restriction of $\varphi$ to $V(F_i)$ is a homomorphism from $F_i$ to $G_i[S]$ multiplied by the probability that, for every $w\in V(H_i)\setminus V(F_i)$, if $v$ is the unique neighbour of $w$ in $H_i$, then $\varphi(w)$ is adjacent to $\varphi(v)$ in $G_i$. Thus,
\[t(H_i,G_i)\geq \left(|S|/n\right)^{v(F_i)}t(F_i,G_i[S])\left(\frac{1+\xi}{\chi(F_j)-1}\right)^{v(H_i)-v(F_i)}>\tau\left(\frac{1+\xi}{\chi(F_j)-1}\right)^{v(H_i)-v(F_i)}\]
\[=\tau(1+\xi)^{t_i}\left(\frac{1}{\chi(F_j)-1}\right)^{v(H_i)-v(F_i)} > 3\left(\frac{1}{\chi(F_j)-1}\right)^{v(H_i)-k(F_i)}\]
where the last inequality follows from \eqref{eq:t0bound1} and the facts that $t_1,t_2\geq t_0$ and $v(F_i)\geq k(F_i)$. So, by \eqref{eq:tinjt} and the fact that $n$ is large, we have
\[t_{\inj}(H_i,G_i) = t(H_i,G_i)-o(1) > 2\left(\frac{1}{\chi(F_j)-1}\right)^{v(H_i)-k(F_i)}.\]
Consequently, by \eqref{eq:samesame},
\[m(H_1,H_2;G_1)=\rho_1\cdot t_{\inj}(H_1,G_1) + \rho_2\cdot t_{\inj}(H_2,G_2) > 2\]
which contradicts \eqref{eq:G1good} and thus completes the proof.
\end{proof}

Next, we prove that $V_0$ is quite small. For this, we assume the following bound on $\tau$:
\begin{equation}
\label{eq:taua}
0<\tau<\frac{c_1(F_1,F_2)\cdot \xi^{f}}{4}.
\end{equation}
Note that $c_1(F_1,F_2)>0$ by~\cite[Lemma~2.11]{MossNoel23++} and so it is possible to choose $\tau$ to satisfy this condition. The next lemma is analogous to~\cite[Claim~3.3]{FoxWigderson23}.

\begin{lem}
\label{lem:V0small}
$|V_0|< \xi n$.
\end{lem}

\begin{proof}
Suppose that $|V_0|\geq \xi n$. Our goal is to obtain a contradiction via an application of Lemma~\ref{lem:inducedSubgraph} with $S=V_0$. By definition of $V_0$, for each $i\in\{1,2\}$, every $v\in V_0$ satisfies 
\[d_i(v)\leq \left(1-\frac{1+2\xi}{\chi(F_i)-1}\right)(n-1).\]
Since $d_1(v)+d_2(v)=n-1$, this tells us that 
\[d_1(v)\geq \left(\frac{1+2\xi}{\chi(F_2)-1}\right)(n-1) > \left(\frac{1+\xi}{\chi(F_2)-1}\right)n\]
and 
\[d_2(v)\geq \left(\frac{1+2\xi}{\chi(F_1)-1}\right)(n-1)> \left(\frac{1+\xi}{\chi(F_1)-1}\right)n\] 
for every vertex $v\in V_0$ and large enough $n$. Therefore, by Lemma~\ref{lem:inducedSubgraph}, for each $i\in \{1,2\}$, we must have
\[t(F_i,G_i[V_0]) \leq \tau(n/|V_0|)^{v(F_i)} \leq \tau/\xi^{v(F_i)}.\]
According to \eqref{eq:taua}, this last expression is less than $c_1(F_1,F_2)/4$. Thus, for $i\in\{1,2\}$,
\[t(F_i,G_i[V_0]) < c_1(F_1,F_2)/4.\]
On the other hand, by definition of $c_1(F_1,F_2)$, we have
\[t(F_1,G_1[V_0]) + t(F_2,G_2[V_0])\geq c_1(F_1,F_2) - o(1)\]
where the $o(1)$ term tends to $0$ as $\xi n\to\infty$. Since $c_1(F_1,F_2)>0$ by~\cite[Lemma~2.11]{MossNoel23++} implies that, for large $n$,
\[t(F_1,G_1[V_0]) + t(F_2,G_2[V_0])\geq c_1(F_1,F_2)/2.\]
Therefore, we can let $i\in \{1,2\}$ such that
\[t(F_i,G_i[V_0])\geq c_1(F_1,F_2)/4.\]
Combining the upper and lower bound on $t(F_i,G_i[V_0])$ that we have proven leads to a contradiction, thereby completing the proof.
\end{proof}

Next, we show that $V_1$ is small which, when combined with the fact that $|V_0|<\xi n$, implies that the vast majority of the vertices are in $V_2$. To do this, we use the  following form of the Erd\H{o}s--Simonovits Supersaturation Theorem.

\begin{thm}[Erd\H{o}s--Simonovits Supersaturation Theorem~\cite{ErdosSimonovits83}]
\label{th:supersat}
For every non-empty graph $F$ and $\xi>0$ there exists $\gamma=\gamma(F,\xi)>0$ such that if $G$ is a graph with $t(K_2,G)\geq 1-\frac{1-\xi}{\chi(F)-1}$, then $t(F,G)\geq \gamma$.
\end{thm}

Using Theorem~\ref{th:supersat}, we define $\gamma$ by
\begin{equation}
\label{eq:gamma}
\gamma:=\min\{\gamma(F_1,\xi),\gamma(F_2,\xi)\}.
\end{equation}
We also assume that $\tau$ is chosen so that
\begin{equation}
\label{eq:taub}
0<\tau<\gamma\cdot (25\xi)^{f}.
\end{equation}
The next lemma is analogous to~\cite[Claim~3.4]{FoxWigderson23}.

\begin{lem}
\label{lem:V1small}
$|V_1|< 25\xi n$. 
\end{lem}

\begin{proof}
Let us begin by establishing the following claim.

\begin{claim}
\label{claim:V12sparse}
For each $i\in \{1,2\}$, if $|V_i|\geq 25\xi n$, then
\[e_i(V_i) < \left(1-\frac{1-\xi}{\chi(F_i)-1}\right)\frac{|V_i|^2}{2}.\]
\end{claim}

\begin{proof}[Proof of Claim~\ref{claim:V12sparse}]
Suppose not. Then there exists $i \in \{1,2\}$ such that
\[ t(K_2,G_i[V_i]) = \frac{2e_i(V_i)}{|V_i|^2} \geq 1-\frac{1-\xi}{\chi(F_i)-1}.\]
Consequently, Theorem~\ref{th:supersat} implies that $t(F_i,G_i[V_i])\geq\gamma$. Using the hypothesis $|V_i|\geq 25\xi n$,
\[\gamma \geq \gamma\cdot \left(\frac{25\xi n}{|V_i|}\right)^{v(F_i)}\geq \gamma\cdot (25\xi)^{v(F_i)}\left(\frac{n}{|V_i|}\right)^{v(F_i)}\]
which, by \eqref{eq:taub}, is greater than $\tau\cdot(n/|V_i|)^{v(F_i)}$. By Observation~\ref{obs:Vihigh}, we have $d_i(v)\geq \left(\frac{1+\xi}{\chi(F_j)-1}\right)n$ for all $v\in V_i$, where $j\in\{1,2\}\setminus\{i\}$. So, the set $S=V_i$ contradicts Lemma~\ref{lem:inducedSubgraph}. Therefore, the claim holds.
\end{proof}

We now use Claim~\ref{claim:V12sparse} to complete the proof of the lemma. If $|V_1|<25\xi n$, then we are done; so, we assume $|V_1|\geq 25\xi n$. By Assumption~\ref{ass:V1<=V2}, $|V_2|\geq 25\xi n$ as well. In particular, both $V_1$ and $V_2$ are non-empty and satisfy the hypothesis, and therefore the conclusion, of Claim~\ref{claim:V12sparse}. For $i\in \{1,2\}$, define
\[\eta_i:=\frac{e_i(V_1,V_2)}{|V_1||V_2|}\]
and note that $\eta_1+\eta_2=1$ because $V_1$ and $V_2$ are disjoint (by Lemma~\ref{lem:V1V2disjoint}) and $G_2$ is the complement of $G_1$. By definition of $V_1$, we have
\[ \sum_{v\in V_1}d_1(v) \geq |V_1|\left(1-\frac{1+2\xi}{\chi(F_1)-1}\right)(n-1) \geq |V_1|\left(1-\frac{1+3\xi}{\chi(F_1)-1}\right)n.\]
On the other hand,
\[\sum_{v\in V_1}d_1(v) = 2e_1(V_1) + e_1(V_1,V\setminus V_1)\]
\[ = 2e_1(V_1) + e_1(V_1,V_2) + e_1(V_1,V_0)\leq 2e_1(V_1)+\eta_1|V_1||V_2| + |V_1||V_0|.\]
Claim~\ref{claim:V12sparse} tells us that the above expression is less than
\[\left(1-\frac{1-\xi}{\chi(F_1)-1}\right)|V_1|^2 +\eta_1 |V_1||V_2| + |V_1||V_0|.\]
Combining the lower and upper bounds on $\sum_{v\in V_1}d_1(v)$ obtained above and cancelling a factor of $|V_1|$, we get
\[\left(1-\frac{1+3\xi}{\chi(F_1)-1}\right)n \leq \left(1-\frac{1-\xi}{\chi(F_1)-1}\right)|V_1| +\eta_1 |V_2| + |V_0|.\]
By Lemma~\ref{lem:V1V2disjoint}, we have $|V_0|+|V_1|+|V_2|=n$ and so this inequality becomes
\[\left(1-\frac{1}{\chi(F_1)-1}-\frac{3\xi}{\chi(F_1)-1}\right)n \leq n-\left(\frac{1-\xi}{\chi(F_1)-1}\right)|V_1| +(\eta_1-1) |V_2|\]
\[=n-\frac{|V_1|}{\chi(F_1)-1}+\frac{\xi|V_1|}{\chi(F_1)-1} +(\eta_1-1) |V_2|.\]
Adding and subtracting $\frac{|V_2|+|V_0|}{\chi(F_1)-1}$ in this final expression and using $|V_0|+|V_1|+|V_2|=n$ again yields
\[n-\frac{|V_1|}{\chi(F_1)-1}+\frac{\xi|V_1|}{\chi(F_1)-1} +(\eta_1-1) |V_2| + \frac{|V_2|+|V_0|}{\chi(F_1)-1} - \frac{|V_2|+|V_0|}{\chi(F_1)-1}\]
\[=n-\frac{n}{\chi(F_1)-1}+\frac{\xi|V_1|}{\chi(F_1)-1} +(\eta_1-1) |V_2| + \frac{|V_2|+|V_0|}{\chi(F_1)-1}.\]
Since $|V_1|\leq n$ trivially and $|V_0|< \xi n$ by Lemma~\ref{lem:V0small}, we get that this last expression is bounded above by
\[n-\frac{n}{\chi(F_1)-1} +\left(\eta_1-1 + \frac{1}{\chi(F_1)-1}\right) |V_2| + \frac{2\xi n}{\chi(F_1)-1}.\]
To recap, the inequality that we have just derived is
\[\left(1-\frac{1}{\chi(F_1)-1}-\frac{3\xi}{\chi(F_1)-1}\right)n\leq n-\frac{n}{\chi(F_1)-1} +\left(\eta_1-1 + \frac{1}{\chi(F_1)-1}\right) |V_2| + \frac{2\xi n}{\chi(F_1)-1}.\]
By rearranging, we get
\begin{equation}
\label{eq:5xi}
\left(1-\eta_1-\frac{1}{\chi(F_1)-1}\right)|V_2| < \frac{5\xi n}{\chi(F_1)-1}.
\end{equation}
Applying the same argument, but with the roles of $(F_1,V_1,\eta_1)$ and $(F_2,V_2,\eta_2)$ reversed, we get that
\begin{equation}
\label{eq:5xi2}
\left(1-\eta_2-\frac{1}{\chi(F_2)-1}\right)|V_1| < \frac{5\xi n}{\chi(F_2)-1}. 
\end{equation}
We now divide the proof into cases depending on the values of $\eta_1$ and $\eta_2$. 

\begin{case}
\label{case:small}
$\eta_i\leq \frac{2}{5}$ for some $i\in \{1,2\}$. 
\end{case}

Let $j\in\{1,2\}\setminus\{i\}$. Since $\chi(F_i)\geq3$, we get the following by applying \eqref{eq:5xi} or \eqref{eq:5xi2}:
\[\frac{5\xi n}{2}\geq \frac{5\xi n}{\chi(F_i)-1}>\left(1-\eta_i-\frac{1}{\chi(F_i)-1}\right)|V_j|\geq \left(1-\frac{2}{5}-\frac{1}{2}\right)|V_j| = \frac{|V_j|}{10}\]
and so $|V_j|< 25\xi n$. Since $|V_1|\leq |V_2|$ by Assumption~\ref{ass:V1<=V2}, this implies that $|V_1|< 25\xi n$. 

\begin{case}
$\eta_1,\eta_2 > \frac{2}{5}$.
\end{case}

Since $\chi(F_1)+\chi(F_2)\geq 7$, we can let $j\in \{1,2\}$ so that $\chi(F_j)\geq4$ and let $i\in\{1,2\}\setminus\{j\}$. Since $\eta_i>\frac{2}{5}$ and $\eta_i+\eta_j=1$, we have $\eta_j<\frac{3}{5}$. Now, by applying \eqref{eq:5xi} or \eqref{eq:5xi2},
\[\frac{5\xi n}{3}\geq \frac{5\xi n}{\chi(F_j)-1} > \left(1-\eta_j-\frac{1}{\chi(F_j)-1}\right)|V_i| > \left(1-\frac{3}{5}-\frac{1}{3}\right)|V_i| = \frac{|V_i|}{15}\]
and so $|V_i|< 25\xi n$. Since $|V_1|\leq |V_2|$ by Assumption~\ref{ass:V1<=V2}, this completes the proof.
\end{proof}

\subsection{Obtaining the Partition}

Now that we know that most vertices are in $V_2$ and, thus, have high degree in $G_2$, the next step is to show that $G_2$ can be made $F_2$-free by deleting a small proportion of its edges. For this, we apply the well-known Graph Removal Lemma of Erd\H{o}s, Frankl and R\"odl~\cite{ErdosFranklRodl86}. As discussed in Remark~\ref{rem:FW}, this is one place in which our argument deviates from that of~\cite{FoxWigderson23}. 

\begin{thm}[Graph Removal Lemma~\cite{ErdosFranklRodl86}]
\label{th:removal}
For every graph $F$ and any given $\delta>0$, there is a $\beta=\beta(F,\delta)>0$ such that if $G$ is a graph with $t(F,G)\leq \beta$, then there is a spanning subgraph $G'$ of $G$ with $t(F,G')=0$ and $t(K_2,G')\geq t(K_2,G)-\delta$.
\end{thm}

Using Theorem~\ref{th:removal}, define
\begin{equation}\label{eq:beta}\beta:=\min\left\{\beta(F_1,\delta/2),\beta(F_2,\delta/2)\right\}.\end{equation}
We also assume that $\xi$ and $\tau$ satisfy the following:
\begin{equation}
\label{eq:xi2}
0<\xi<\frac{\beta}{52\cdot f},
\end{equation}
\begin{equation}
\label{eq:tauc}
0<\tau<\beta/2.
\end{equation}
The following lemma allows us to apply Theorem~\ref{th:removal}.

\begin{lem}
\label{lem:F2small}
$t(F_2,G_2)< \beta$.
\end{lem}

\begin{proof}
Suppose, to the contrary, that $t(F_2,G_2)\geq\beta$. The number of homomorphisms from $F_2$ to $G_2$ which map a vertex to $V_0\cup V_1$ is at most $v(F_2)\cdot n^{v(F_2)-1}\cdot |V_0\cup V_1|$ which, by Lemmas~\ref{lem:V0small} and~\ref{lem:V1small}, is no more than $26v(F_2)\xi n^{v(F_2)}$. Therefore,
\[t(F_2,G_2[V_2]) \geq  \frac{t(F_2,G_2)n^{v(F_2)} - 26v(F_2)\xi n^{v(F_2)}}{|V_2|^{v(F_2)}}\geq (n/|V_2|)^{v(F_2)}(\beta-26v(F_2)\xi).\]
By \eqref{eq:xi2} and \eqref{eq:tauc}, this is greater than $\tau\cdot (n/|V_2|)^{v(F_2)}$. Recall that, by Observation~\ref{obs:Vihigh}, we have $d_2(v)\geq \left(\frac{1+\xi}{\chi(F_1)-1}\right)n$ for all $v\in V_2$. So, the set $S=V_2$ contradicts Lemma~\ref{lem:inducedSubgraph}, and the proof is complete. 
\end{proof}

The last step in verifying Lemma~\ref{lem:roughStructure} involves utilizing the Erd\H{o}s--Simonovits Stability Theorem~\cite{Simonovits68} in the following form. 

\begin{thm}[Erd\H{o}s--Simonovits Stability Theorem~\cite{Simonovits68}]
\label{th:stability}
For every non-empty graph $F$ and $\varepsilon>0$, there exists $\delta=\delta(F,\varepsilon)>0$ such that if $G$ is a graph with $t(F,G)=0$ and $t(K_2,G)\geq 1-\frac{1}{\chi(F)-1}-\delta$, then there exists a partition $A_1,\dots,A_{\chi(F)-1}$ of $V(G)$ such that $\sum_{i=1}^{\chi(F)-1}e(A_i)\leq \varepsilon n^2$.
\end{thm}

Using Theorem~\ref{th:stability}, we define 
\begin{equation}\label{eq:delta}\delta:=\min\left\{\delta(F_1,\varepsilon/2),\delta(F_2,\varepsilon/2),2\varepsilon\right\}\end{equation}
We also assume that $\xi$ satisfies
\begin{equation}
\label{eq:xi3}
0<\xi<\frac{\delta}{52}.
\end{equation}
Finally, we present the proof of Lemma~\ref{lem:roughStructure}, thereby accomplishing our primary objective of this section. 

\begin{proof}[Proof of Lemma~\ref{lem:roughStructure}]
First, let us bound $t(K_2,G_2)$ from below. By Lemmas~\ref{lem:V0small} and~\ref{lem:V1small}, we have $|V_2|=n-|V_0\cup V_1|\geq (1-26\xi )n$. Therefore, 
\[t(K_2,G_2)=\frac{2e(G_2)}{n^2} = \frac{1}{n^2}\sum_{v\in V}d_2(v)\geq \frac{1}{n^2}\sum_{v\in V_2}d_2(v)\geq \frac{|V_2|}{n^2}\left(1-\frac{1+2\xi}{\chi(F_2)-1}\right)(n-1)\]
\[\geq \frac{|V_2|}{n^2}\left(1-\frac{1+3\xi}{\chi(F_2)-1}\right)n\geq (1-26\xi)\left(1-\frac{1+3\xi}{\chi(F_2)-1}\right)\geq 1-\frac{1}{\chi(F_2)-1}-26\xi.\]
By \eqref{eq:xi3}, this is at least $1-\frac{1}{\chi(F_2)-1}-\frac{\delta}{2}$. So,
\begin{equation}
\label{eq:G2dense}
t(K_2,G_2)\geq 1-\frac{1}{\chi(F_2)-1}-\frac{\delta}{2}.
\end{equation}

Now, by Lemma~\ref{lem:F2small}, Theorem~\ref{th:removal} and \eqref{eq:beta}, there exists a spanning subgraph $G_2'$ of $G_2$ such that $t(F_2,G_2')=0$ and $t(K_2,G_2')\geq t(K_2,G_2)-\delta/2$. So, \eqref{eq:G2dense} implies that $t(K_2,G_2')\geq 1-\frac{1}{\chi(F_2)-1} - \delta$. By Theorem~\ref{th:stability} and \eqref{eq:delta}, there is a partition $A_1,A_2,\dots, A_{\chi(F_2)-1}$ of $V$ such that $\sum_{i=1}^{\chi(F_2)-1}e_{G_2'}(A_i)\leq (\varepsilon/2)n^2$. Since $t(K_2,G)=2e(G)/v(G)^2$ for any graph $G$, the inequality $t(K_2,G_2')\geq t(K_2,G_2)-\delta/2$ is equivalent to $e(G_2')\geq e(G_2) - (\delta/4)n^2$. Therefore, 
\[\sum_{i=1}^{\chi(F_2)-1}e_{2}(A_i)\leq \sum_{i=1}^{\chi(F_2)-1}e_{G_2'}(A_i) + (\delta/4)n^2 \leq (\varepsilon/2+\delta/4)n^2\leq\varepsilon n^2\]
where the last inequality follows from \eqref{eq:delta}. 
\end{proof}

\section{Proof of Theorem~\ref{th:hairy}: Exact Structure}
\label{sec:hairier}

The aim of this section is to complete the proof of Theorem~\ref{th:hairy}. The way that this breaks down is as follows. We start by obtaining control over the number of copies of $H_1$ in $G_1$ and $H_2$ in $G_2$ that contain any given vertex $v\in V$. In particular, we show that any two vertices in $V$ ``contribute'' roughly the same amount to $m(H_1,H_2;G_1)$. Thus, if one vertex contributes ``too much,'' then all vertices do, which leads to a violation of \eqref{eq:G1good}.

After this, we refine the rough structure afforded to us by Lemma~\ref{lem:roughStructure} until we get that $G_2$ is simply a Tur\'an graph with $\chi(F_2)-1$ parts. The first step in this process is to show that the parts have nearly the same size and almost all edges between pairs of parts are in $G_2$. We then show that the $G_2$-neighbourhood of every vertex $v\notin V_1$ roughly ``respects'' the partition. Next, we prove that the $G_2$-degree of any vertex is within a small window around $\left(1-\frac{1}{\chi(F_2)-1}\right)n$, which implies that $V_1=\emptyset$. After that, we can use the critical edge in $F_2$ to show that all edges within $A_i$ must be in $G_1$ for all $1\leq i\leq\chi(F_2)-1$. From this point, the theorem is easily deduced via a convexity argument.

\subsection{Every Vertex Contributes the Same}

For a graph $H$, a graph $G$ on vertex set $V$ and $v\in V$, define $t_{\inj}(H,G)(v)$ to be the probability that a random function from $V(H)$ to $V$ is an injective homomorphism from $H$ to $G$ whose image contains $v$. The main idea of the next lemma is that, if $u$ and $w$ are vertices such that an appropriate weighted sum of $t_{\inj}(H_1,G_1)(u)$ and $t_{\inj}(H_2,G_2)(u)$ is significantly smaller than the analogous sum for $w$, then one can get a better colouring by ``deleting'' $w$ and ``cloning'' $u$. This is a standard idea in extremal combinatorics going back at least as far as Zykov's proof of Tur\'an's Theorem~\cite{Zykov49}. This lemma is analogous to~\cite[Lemma~2.1]{FoxWigderson23}.

\begin{lem}
\label{lem:allVertsSame}
There exists a constant $C=C(H_1,H_2)>0$ such that, for any $u,w\in V$,
\[\rho_1\cdot t_{\inj}(H_1,G_1)(u) + \rho_2\cdot t_{\inj}(H_2,G_2)(u)\geq \rho_1\cdot t_{\inj}(H_1,G_1)(w) + \rho_2\cdot t_{\inj}(H_2,G_2)(w)-\frac{C}{n^2}.\]
\end{lem}

\begin{proof}
Suppose, to the contrary, that the inequality does not hold for some $u,w\in V$. If we remove all edges incident to the vertex $w$ from $G_1$, then we lose all of the injective homomorphisms from $H_1$ to $G_1$ which map at least one vertex to $w$. Likewise, if we delete all edges incident to $w$ from $G_2$, then we lose all of the injective homomorphisms from $H_2$ to $G_2$ which map at least one vertex to $w$. (Note that, here, we are subtly using the assumption that the graphs $H_1$ and $H_2$ have no singleton components.) 

After deleting all such edges from $G_1$ and $G_2$, suppose that we add to $G_1$ all edges of the form $wv$ such that $uv\in E(G_1)$ and $v\neq w$ to form a new graph $G_1'$. Similarly, add to $G_2$ all edges of the form $wv$ such that $v$ is a vertex with $uv\notin E(G_1)$ and $v\neq w$ to get a graph $G_2'$. Note that $G_2'$ is the complement of $G_1'$. In adding these edges, we gain one injective homomorphism from $H_1$ to $G_1'$ per injective homomorphism from $H_1$ to $G_1$ that includes $u$ and not $w$. Similarly, we gain one injective homomorphsim from $H_2$ to $G_2'$ per injective homomorphism from $H_2$ to $G_2$ that includes $u$ and not $w$. Additionally, for each $i\in\{1,2\}$, we may also gain $O(n^{v(H_i)-2})$ injective homomorphisms which map to both $u$ and $w$. Thus, for each $i\in \{1,2\}$,
\[t_{\inj}(H_i,G_i') \leq t_{\inj}(H_i,G_i)-t_{\inj}(H_i,G_i)(w)+t_{\inj}(H_i,G_i)(u) + O(1/n^2)\]
where the constant factor on the $O(1/n^2)$ term is bounded by a function of $H_i$. Thus, assuming that the inequality in the lemma is not true, we have that $m(H_1,H_2;G_1')$ is at most $m(H_1,H_2;G_1)$ plus a $O(1/n^2)$ term, where the constant factor depends on $H_1$ and $H_2$, minus $C/n^2$. So, if $C$ is chosen large enough with respect to $H_1$ and $H_2$, we get that $G_1'$ contradicts our choice  of $G_1$. Thus, the lemma holds.
\end{proof}

Analogous to the definition of $t_{\inj}(H,G)(v)$, let $t(H,G)(v)$ be the probability that a uniformly random function from $V(H)$ to $V$ is a homomorphism from $H$ to $G$ whose image contains $v$. The following lemma restricts $t(H_i,G_i)(v)$ for every vertex $v$. 

\begin{lem}
\label{lem:contribsmall}
Suppose that $\{i,j\}=\{1,2\}$. For every $v\in V$, 
\[t(H_i,G_i)(v)\leq\frac{3\max\{v(H_1),v(H_2)\}}{n}\left(\frac{1}{\chi(F_j)-1}\right)^{v(H_i)-k(F_i)}.\]
\end{lem}

\begin{proof}
Suppose, to the contrary, that there exists $v\in V$ such that
\begin{equation}\label{eq:thigibound}t(H_i,G_i)(v)>\frac{3\max\{v(H_1),v(H_2)\}}{n}\left(\frac{1}{\chi(F_j)-1}\right)^{v(H_i)-k(F_i)}.\end{equation}
By Lemma~\ref{lem:allVertsSame}, \eqref{eq:samesame}, \eqref{eq:thigibound} and the fact that $t_{\inj}(H_i,G_i)(v)=t(H_i,G_i)(v)+O(1/n^2)$, we get that, for large $n$, every $u\in V$ satisfies
\[\rho_1\cdot t_{\inj}(H_1,G_1)(u)+\rho_2\cdot t_{\inj}(H_2,G_2)(u)> \frac{2\max\{v(H_1),v(H_2)\}}{n}.\]
Summing this inequality over all $u\in V$ yields
\[2\max\{v(H_1),v(H_2)\}  < \sum_{u\in V}\left(\rho_1\cdot t_{\inj}(H_1,G_1)(u)+\rho_2\cdot t_{\inj}(H_2,G_2)(u)\right)\]
\[=\rho_1\cdot v(H_1)\cdot t_{\inj}(H_1,G_1) + \rho_2\cdot v(H_2)\cdot t_{\inj}(H_2,G_2)\]
\[\leq \max\{v(H_1),v(H_2)\}\left(\rho_1\cdot t_{\inj}(H_1,G_1) + \rho_2\cdot t_{\inj}(H_2,G_2)\right).\]
This contradicts \eqref{eq:G1good}, and thus the proof is complete. 
\end{proof}

\subsection{Refining the Partition}

We assume, throughout the remainder of this section, that $A_1,\dots, A_{\chi(F_2)-1}$ is a partition of $V$ as in Lemma~\ref{lem:roughStructure}. Let us show that the sets $A_1,\dots,A_{\chi(F_2)-1}$ have approximately the same size and that $G_1$ contains almost no edges between different parts. In order to prove this, we make the following assumption on $\varepsilon$. Recall that $\chi = \max\{\chi(F_1), \chi(F_2)\}$.
\begin{equation}\label{eq:epsilon1}
0<\varepsilon<\frac{1}{12\chi^4}.
\end{equation}
The next lemma is analogous to~\cite[Claim~3.8]{FoxWigderson23}.

\begin{lem}
\label{lem:partitionProperties}
For $1\leq i\neq j\leq \chi(F_2)-1$,
\begin{enumerate}
\stepcounter{equation}
    \item\label{eq:similarSizes} $\left||A_i|-\frac{n}{\chi(F_2)-1}\right|\leq \sqrt{3\varepsilon} \cdot n$ and
\stepcounter{equation}
    \item\label{eq:denseBetween} $e_2(A_i,A_j)\geq (1-2\chi(F_2)^2\varepsilon)|A_i||A_j|$.
\end{enumerate}
\end{lem}

\begin{proof}
First observe that, since $\sum_{i=1}^{\chi(F_2)-1}|A_i|=n$,
\[\sum_{i=1}^{\chi(F_2)-1}\left(\frac{|A_i|}{n}-\frac{1}{\chi(F_2)-1}\right)^2 = \sum_{i=1}^{\chi(F_2)-1}\frac{|A_i|^2}{n^2} - 2\sum_{i=1}^{\chi(F_2)-1}\frac{|A_i|}{n(\chi(F_2)-1)} + \sum_{i=1}^{\chi(F_2)-1}\left(\frac{1}{\chi(F_2)-1}\right)^2\]
\[=\sum_{i=1}^{\chi(F_2)-1}\frac{|A_i|^2}{n^2} - \frac{1}{\chi(F_2)-1}\]
and, also,
\[1=\left(\sum_{i=1}^{\chi(F_2)-1}\frac{|A_i|}{n}\right)^2 = \sum_{i=1}^{\chi(F_2)-1}\frac{|A_i|^2}{n^2} + 2\left(\sum_{1\leq i< j\leq \chi(F_2)-1}\frac{|A_i||A_j|}{n^2}\right).\]
Solving for $\sum_{i=1}^{\chi(F_2)-1}\frac{|A_i|^2}{n^2}$ in one of these two equations and substituting into the other yields
\[\sum_{i=1}^{\chi(F_2)-1}\left(\frac{|A_i|}{n}-\frac{1}{\chi(F_2)-1}\right)^2 + \frac{1}{\chi(F_2)-1} = 1-2\left(\sum_{1\leq i< j\leq \chi(F_2)-1}\frac{|A_i||A_j|}{n^2}\right)\]
which is equivalent to
\begin{equation}
    \label{eq:muthing}
1-\frac{1}{\chi(F_2)-1} = \sum_{i=1}^{\chi(F_2)-1}\left(\frac{|A_i|}{n}-\frac{1}{\chi(F_2)-1}\right)^2 +2\left(\sum_{1\leq i< j\leq \chi(F_2)-1}\frac{|A_i||A_j|}{n^2}\right).
\end{equation}
Also, by \eqref{eq:delta} and \eqref{eq:G2dense} (which remains true outside of the context of the proof of Lemma~\ref{lem:roughStructure}), we have $t(K_2,G_2)\geq 1-\frac{1}{\chi(F_2)-1}-\varepsilon$. So,
\[1-\frac{1}{\chi(F_2)-1}-\varepsilon \leq t(K_2,G_2) =\frac{2e(G_2)}{n^2} = \sum_{i=1}^{\chi(F_2)-1}\frac{2e_2(A_i)}{n^2} + \sum_{1\leq i< j\leq \chi(F_2)-1}\frac{2e_2(A_i,A_j)}{n^2}\]
\[\leq 2\varepsilon + 2\left(\sum_{1\leq i< j\leq \chi(F_2)-1}\frac{e_2(A_i,A_j)}{n^2}\right)\]
where the last inequality is by Lemma~\ref{lem:roughStructure}. Substituting the expression for $1-\frac{1}{\chi(F_2)-1}$ in \eqref{eq:muthing} into this inequality and rearranging yields
\[\sum_{i=1}^{\chi(F_2)-1}\left(\frac{|A_i|}{n}-\frac{1}{\chi(F_2)-1}\right)^2 + 2\sum_{1\leq i< j\leq \chi(F_2)-1}\left(\frac{|A_i||A_j|}{n^2}-\frac{e_2(A_i,A_j)}{n^2}\right) \leq 3\varepsilon.\]
Since all summands on the left side are non-negative, we get
\[\left(\frac{|A_i|}{n}-\frac{1}{\chi(F_2)-1}\right)^2\leq 3\varepsilon\]
for all $i$, which proves \eqref{eq:similarSizes}. Similarly, for each $i\neq j$, the above inequality implies that 
\[2\left(\frac{|A_i||A_j|}{n^2}-\frac{e_2(A_i,A_j)}{n^2}\right) \leq 3\varepsilon\]
and so 
\[e_2(A_i,A_j)\geq |A_i||A_j|-(3\varepsilon/2)n^2 = \left(1-\frac{3\varepsilon n^2}{2|A_i||A_j|}\right)|A_i||A_j|.\]
By \eqref{eq:similarSizes}, the right side is at least
\[\left(1-\frac{3\varepsilon }{2\left(\frac{1}{\chi(F_2)-1}-\sqrt{3\varepsilon}\right)^2}\right)|A_i||A_j|\]
 and by \eqref{eq:epsilon1}, this is at least $(1-2\chi(F_2)^2\varepsilon)|A_i||A_j|$ (with room to spare). Therefore, \eqref{eq:denseBetween} holds.
\end{proof}

Next, we show that the $G_2$-neighbourhood of every vertex that is not in $V_1$ roughly ``respects'' the partition $A_1,\dots,A_{\chi(F_2)-1}$ (see Lemma~\ref{lem:noBigToAll} below). We assume that $\varepsilon$ satisfies the following condition:
\begin{equation}
\label{eq:epsilon2}
0<\varepsilon<\frac{\theta^2}{4\cdot\chi^2f^2}.
\end{equation}
Also, we assume that $t_0$ is chosen large enough so that, for all $t\geq t_0$, we have
\begin{equation}
\label{eq:t0bound2}
\left(\frac{5}{4}\right)^{t/2} > \frac{6(t+f)}{q\left(\frac{\theta}{2\chi}\right)^{f}}.\end{equation}

\begin{defn}
\label{defn:bad}
Say that a vertex $v\in V$ is \emph{bad} if, for all $1\leq i\leq \chi(F_2)-1$, the number of $G_2$-neighbours of $v$ in $A_i\cap V_2$  is at least $\theta |A_i|$. Let $B$ be the set of all bad vertices. 
\end{defn}

\begin{lem}
\label{lem:noBigToAll}
$B\subseteq V_1$. 
\end{lem}

\begin{proof}
Suppose, to the contrary, that there exists a bad vertex $v\notin V_1$. For each vertex $x$ of $F_2$, let $p(x)$ be the number of pendant edges incident to $x$ which were added during the construction of $H_2$ from $F_2$. Let $v_0z$ be a critical edge of $F_2$ where we assume, without loss of generality, that $p(v_0)\leq p(z)$. Then, in particular, at most half of the pendant edges added in the construction of $H_2$ are incident to $v_0$; i.e. $p(v_0)\leq \frac{v(H_2)-v(F_2)}{2}$. Let $F_2':=F_2\setminus \{v_0\}$ and note that, since $v_0z$ is a critical edge, $\chi(F_2')<\chi(F_2)$. Fix a proper colouring $\psi:V(F_2')\to\{1,\dots, \chi(F_2)-1\}$ of $F_2'$. Our aim is to prove a lower bound on $t(H_2,G_2)(v)$ which is large enough to contradict Lemma~\ref{lem:contribsmall}.

Let $S$ be the neighbourhood of $v$ in $G_2$. The probability that a random function $\varphi$ from $V(H_2)$ to $V$ is a homomorphism from $H_2$ to $G_2$ such that $\varphi(v_0)=v$ is at least the probability that the restriction of $\varphi$ to $V(F_2')$ is a homomorphism from $F_2'$ to $G_2[S\cap V_2]$, times $1/n$ (the probability that $\varphi(v_0)=v$), times the probability that every vertex of $V(H_2)\setminus V(F_2)$ is mapped to a $G_2$-neighbour of the image of its unique neighbour in $H_2$. Taking into account that $d_2(w)\geq \left(\frac{5}{4}\cdot\frac{1}{\chi(F_1)-1}\right)n$ for all $w\in V_2$ by Observation~\ref{obs:Vihigh} and that $v\notin V_1$, we have
\[d_2(v)=n-1-d_1(v)\geq \left(\frac{1+2\xi}{\chi(F_1)-1}\right)(n-1) > \left(\frac{1}{\chi(F_1)-1}\right)n.\]
Thus, $t(H_2,G_2)(v)$ is greater than
\begin{equation}\label{eq:F2'}\left(\frac{|S\cap V_2|}{n}\right)^{v(F_2')}t(F_2',G[S\cap V_2]) \frac{1}{n} \left(\frac{1}{\chi(F_1)-1}\right)^{p(v_0)}\left(\frac{5}{4}\cdot\frac{1}{\chi(F_1)-1}\right)^{v(H_2)-v(F_2)-p(v_0)}.\end{equation}

Next, we bound $\left(|S\cap V_2|/n\right)^{v(F_2')}t(F_2',G[S\cap V_2])$  from below. First, since $v$ is bad, we have that $|S\cap A_i\cap V_2|\geq \theta|A_i|$ for all $1\leq i\leq \chi(F_2)-1$. So, if we map $V(F_2')$ randomly to $V$, then the probability that every vertex $w$ of $F_2'$ is mapped to $S\cap A_{\psi(w)}\cap V_2$ is at least $\prod_{w\in V(F_2')}(\theta|A_{\psi(w)}|/n)$ which, by \eqref{eq:similarSizes}, is at least $\theta^{v(F_2')}\left(\frac{1}{\chi(F_2)-1}-\sqrt{3\varepsilon}\right)^{v(F_2')}$. By \eqref{eq:denseBetween}, the number of non-edges in $G_2$ from $S\cap A_i\cap V_2$ to $S\cap A_j\cap V_2$ for $i\neq j$ is at most $2\chi(F_2)^2\varepsilon|A_i||A_j|$ which, since $v$ is bad, is at most
\[\frac{2\chi(F_2)^2\varepsilon|S\cap A_i\cap V_2||S\cap A_j\cap V_2|}{\theta^2}.\]
Thus, for any fixed edge $wy$ of $F_2'$, the conditional probability that $\varphi(w)$ is not adjacent to $\varphi(y)$ given that $\varphi(w)\in S\cap A_{\psi(w)}\cap V_2$ and $\varphi(y)\in S\cap A_{\psi(y)}\cap V_2$ is at most $\frac{2\chi(F_2)^2\varepsilon}{\theta^2}$. By taking a union bound over all edges of $F_2'$, we get that the probability that every vertex $w$ of $F_2'$ is mapped to $S\cap A_{\psi(w)}\cap V_2$ and no edge of $F_2'$ is mapped to a non-edge of $G_2$ is at least
\[\theta^{v(F_2')}\left(\frac{1}{\chi(F_2)-1}-\sqrt{3\varepsilon}\right)^{v(F_2')}\left(1-\frac{2e(F_2')\chi(F_2)^2\varepsilon}{\theta^2}\right).\]
By \eqref{eq:epsilon1}, the product of the first two factors is at least $\left(\frac{\theta}{\chi(F_2)}\right)^{v(F_2)}$ and, by \eqref{eq:epsilon2}, the third factor is at least $1/2$.  So, the expression in \eqref{eq:F2'} is at least
\[\frac{1}{2}\left(\frac{\theta}{\chi(F_2)}\right)^{v(F_2)}\frac{1}{n} \left(\frac{1}{\chi(F_1)-1}\right)^{p(v_0)}\left(\frac{5}{4}\cdot\frac{1}{\chi(F_1)-1}\right)^{v(H_2)-v(F_2)-p(v_0)}\]
\[= \frac{1}{2}\left(\frac{\theta}{\chi(F_2)}\right)^{v(F_2)}\frac{1}{n} \left(\frac{4}{5}\right)^{p(v_0)}\left(\frac{5}{4}\cdot\frac{1}{\chi(F_1)-1}\right)^{v(H_2)-v(F_2)}.\]
Now, since $p(v_0)\leq \frac{v(H_2)-v(F_2)}{2}$, we get that this is at least
\[\frac{1}{2}\left(\frac{\theta}{\chi(F_2)}\right)^{v(F_2)}\frac{1}{n} \left(\frac{4}{5}\right)^{\frac{v(H_2)-v(F_2)}{2}}\left(\frac{5}{4}\cdot\frac{1}{\chi(F_1)-1}\right)^{v(H_2)-v(F_2)}\]
\[= \frac{1}{2}\left(\frac{\theta}{\chi(F_2)}\right)^{v(F_2)}\frac{1}{n} \left(\frac{5}{4}\right)^{\frac{v(H_2)-v(F_2)}{2}}\left(\frac{1}{\chi(F_1)-1}\right)^{v(H_2)-v(F_2)}\]
\[= \frac{1}{2}\left(\frac{\theta}{\chi(F_2)}\right)^{v(F_2)}\frac{1}{n} \left(\frac{5}{4}\right)^{t_2/2}\left(\frac{1}{\chi(F_1)-1}\right)^{v(H_2)-v(F_2)}.\]
By \eqref{eq:t0bound2} and the fact that $t_2\geq t_0$, this is at least
\[\frac{1}{2}\left(\frac{\theta}{\chi(F_2)}\right)^{v(F_2)}\frac{1}{n} \left(\frac{6v(H_2)}{q\left(\frac{\theta}{\chi(F_2)}\right)^{v(F_2)}}\right)\left(\frac{1}{\chi(F_1)-1}\right)^{v(H_2)-v(F_2)}\]
\[> \frac{3v(H_2)}{q\cdot n}\left(\frac{1}{\chi(F_1)-1}\right)^{v(H_2)-k(F_2)}\geq \frac{3\max\{v(H_1),v(H_2)\}}{n}\left(\frac{1}{\chi(F_1)-1}\right)^{v(H_2)-k(F_2)}\]
where the penultimate step uses $v(F_2)> k(F_2)$ and the last step uses \eqref{eq:relSizes}. This contradicts Lemma~\ref{lem:contribsmall} and completes the proof. 
\end{proof}

Using the above lemma, it follows relatively easily that $d_2(v)$ cannot be too large for any vertex $v \in V$. To verify this, we use the following assumption:
\begin{equation}
\label{eq:xi4}
0<\xi< \frac{\theta}{26\chi}.
\end{equation}

\begin{lem}
\label{lem:noTinyRed}
For every $v\in V$,
\[d_2(v)\leq \left(1-\frac{1-3\theta}{\chi(F_2)-1}\right)(n-1).\]
\end{lem}

\begin{proof}
If $v\in V_1$, then, by Observation~\ref{obs:Vihigh},
\[d_2(v)=n-1-d_1(v)\leq n-\left(\frac{5}{4}\cdot \frac{1+\xi}{\chi(F_2)-1}\right)n <\left(1-\frac{1}{\chi(F_2)-1}\right)n< \left(1-\frac{1-3\theta}{\chi(F_2)-1}\right)(n-1).\]
On the other hand, if $v\notin V_1$, then, by Lemma~\ref{lem:noBigToAll}, we have that $v\notin B$. So, there exists $i$ such that $v$ has fewer than $\theta|A_i|$ neighbours in $A_i\cap V_2$. Since $|V\setminus V_2|\leq 26\xi n$ by Lemmas~\ref{lem:V1V2disjoint},~\ref{lem:V0small} and~\ref{lem:V1small}, we have
\[d_2(v)\leq \sum_{j\neq i}|A_j| + \theta\cdot |A_i| + |A_i\setminus V_2|\leq n -(1-\theta)|A_i| + 26\xi n\]
By \eqref{eq:similarSizes} this is at most
\[ n -(1-\theta)\left(\frac{1}{\chi(F_2)-1} - \sqrt{3\varepsilon}\right)n + 26\xi n \leq  \left(1-\frac{1-\theta}{\chi(F_2)-1} + \sqrt{3\varepsilon} + 26\xi\right)n.\]
Note that \eqref{eq:epsilon2} implies that $\varepsilon<\frac{\theta^2}{3(\chi(F_2)-1)^2}$. Using this bound, together with \eqref{eq:xi4}, tells us that the above expression is at most $\left(1-\frac{1-3\theta}{\chi(F_2)-1}\right)(n-1)o$ as desired.
\end{proof}

Next, we show that $d_2(v)$ is reasonably large for every vertex $v\in V$. This will then be used to show that $V_1=\emptyset$. To prove it, we assume the following:
\begin{equation}
\label{eq:epsilon3}
0<\varepsilon<\frac{1}{4f^2\chi^4}.
\end{equation}
Also, we assume that $t_0$ is chosen large enough that, for all $t\geq t_0$,
\begin{equation}
\label{eq:t0bound3}
e^{\theta\cdot t}\geq \frac{6\chi^{2f}(t+f)}{q}.
\end{equation}

\begin{lem}
\label{lem:noLargeRed}
For every $v\in V$,
\[d_2(v)\geq \left(1-\frac{1+15\theta f}{\chi(F_2)-1}\right)(n-1).\]
\end{lem}

\begin{proof}
Suppose that the lemma is false. Then there exists $v\in V$ such that
\[d_1(v)=n-1-d_2(v)> n-1-\left(1-\frac{1+15\theta f}{\chi(F_2)-1}\right)(n-1)\]
\[\geq \left(\frac{1+14\theta f}{\chi(F_2)-1}\right)n.\]
Our goal is to show that $v$ is contained in a large number of copies of $H_1$ in $G_1$ which will contradict Lemma~\ref{lem:contribsmall}. For each $w\in V(F_1)$, let $p(w)$ be the number of pendant edges incident to $w$ added in the construction of $H_1$ from $F_1$ and let $v_0\in V(F_1)$ so that $p(v_0)$ is maximum. Then, by the Pigeonhole Principle,
\begin{equation}\label{eq:p(v_0)}p(v_0)\geq \frac{v(H_1)-v(F_1)}{v(F_1)}.\end{equation} 
Let $F_1':=F_1\setminus \{v_0\}$. 

By the lower bound on $d_1(v)$ proven above, there must exist some $1\leq i\leq \chi(F_2)-1$ such that $v$ has at least
\[\frac{1}{\chi(F_2)-1}\cdot \left(\frac{1+14\theta f}{\chi(F_2)-1}\right)n \geq \frac{n}{(\chi(F_2)-1)^2}\]
neighbours in $A_i$. Let $S$ be the set of neighbours of $v$ in $A_i$. Recall that, by Lemma~\ref{lem:roughStructure}, the number of non-edges of $G_1$ in $S$ is at most
\[\varepsilon n^2 \leq \varepsilon n^2\left(\frac{|S|}{n/(\chi(F_2)-1)^2}\right)^2 = \varepsilon (\chi(F_2)-1)^4|S|^2.\]
Therefore, for large enough $n$,
\[t(K_2,G_1[S]) = \frac{2e_1(S)}{|S|^2} \geq \frac{2\binom{|S|}{2} - 2\varepsilon (\chi(F_2)-1)^4|S|^2}{|S|^2}\geq 1-2\varepsilon \chi(F_2)^4.\]
Thus, if $V(F_1')$ is mapped to $S$ randomly, then the probability that any individual edge of $F_1'$ is mapped to a non-edge of $G_1$ is at most $2\varepsilon\chi(F_2)^4$. So, by a union bound and \eqref{eq:epsilon3}, we have that
\begin{equation}\label{eq:unionBoundF1}t(F_1',G_1[S])\geq 1-2\varepsilon e(F_1')\chi(F_2)^4>1/2.\end{equation}
Now, if $\varphi$ is a random function from $V(H_1)$ to $V$, then the probability that $\varphi$ is a homomorphism mapping $v_0$ to $v$ is at least the probability that the restriction of $\varphi$ to $V(F_1')$ is a homomorphism from $F_1'$ to $G[S]$, times $1/n$ (the probability that $v_0$ maps to $v$) times the probability that every vertex of $V(H_1)\setminus V(F_1)$ is mapped to a $G_1$-neighbour of the image of its unique neighbour in $H_1$. So, by Lemma~\ref{lem:noTinyRed}, $t(H_1,G_1)(v)$ is at least
\[\left(|S|/n\right)^{v(F_1')}t(F_1',G_1[S])\cdot \frac{1}{n}\left(\frac{1+14\theta f}{\chi(F_2)-1}\right)^{p(v_0)}\left(\frac{1-3\theta}{\chi(F_2)-1}\right)^{v(H_1)-v(F_1)-p(v_0)}.\]
Using the fact that $|S|\geq n/(\chi(F_2)-1)^2 > n/\chi(F_2)^2$ and \eqref{eq:unionBoundF1}, we get that this is at least
\[\frac{1}{2\chi(F_2)^{2v(F_1)}}\frac{1}{n}\left(\frac{1}{\chi(F_2)-1}\right)^{v(H_1)-v(F_1)}(1+14\theta f)^{p(v_0)}(1-3\theta)^{v(H_1)-v(F_1)-p(v_0)}.\]
Using the inequalities $1+r\geq e^{r/2}$ and $1-r\geq e^{-2r}$, which are valid for all $r\in [0,1/2]$, we can bound the product of the last two factors as follows:
\[(1+14\theta f)^{p(v_0)}(1-3\theta)^{v(H_1)-v(F_1)-p(v_0)}\]
\[\geq \exp(7\theta f p(v_0) - 6\theta(v(H_1)-v(F_1)-p(v_0))).\]
By \eqref{eq:p(v_0)}, this is at least
\[\exp(7\theta(v(H_1)-v(F_1)) - 6\theta(v(H_1)-v(F_1)-p(v_0))) \geq e^{\theta t_1}.\]
So, by \eqref{eq:t0bound3} and the fact that $t_1\geq t_0$, we have that 
\[t(H_1,G_1)(v) > \frac{1}{n}\left(\frac{1}{\chi(F_2)-1}\right)^{v(H_1)-v(F_1)}\left(\frac{3\left(t_1+f\right)}{q}\right)\]
\[> \frac{3v(H_1)}{q\cdot n}\left(\frac{1}{\chi(F_2)-1}\right)^{v(H_1)-k(F_1)}\geq \frac{3\max\{v(H_1),v(H_2)\}}{n}\left(\frac{1}{\chi(F_2)-1}\right)^{v(H_1)-k(F_1)}\]
where the penultimate step used $v(F_1)>k(F_1)$ and the last step applied \eqref{eq:relSizes}. This contradicts Lemma~\ref{lem:contribsmall} and completes the proof.
\end{proof}

As a consequence of the previous lemma, we will show next that $V_1=\emptyset$. This also implies $B=\emptyset$ by virtue of Lemma~\ref{lem:noBigToAll}. For this, we assume
\begin{equation}\label{eq:theta}
0<\theta<\frac{1}{60 f}.\end{equation}

\begin{lem}
\label{lem:noV1}
We have $V_1=\emptyset$. Consequently, $B=\emptyset$.
\end{lem}

\begin{proof}
Assuming $v\in V_1$, Lemma~\ref{lem:noLargeRed} implies that 
\[d_1(v)=n-1-d_2(v)\leq \left(\frac{1+15\theta f}{\chi(F_2)-1}\right)(n-1)\]
which, by \eqref{eq:theta}, is less than 
\[\left(\frac{5}{4}\cdot \frac{1}{\chi(F_2)-1}\right)n.\] 
This contradicts Observation~\ref{obs:Vihigh}, and so $V_1$ must be empty. Lemma~\ref{lem:noBigToAll} then implies that $B$ is also empty.
\end{proof}

From here forward, we impose an additional assumption that $\sum_{i=1}^{\chi(F_2)-1}e_2(A_i)$ is minimum among all partitions $A_1,\dots,A_{\chi(F_2)-1}$ of $V$. This allows us to prove the next lemma, which is analogous to~\cite[Claim~3.11]{FoxWigderson23}. We assume that $\xi$ satisfies
\begin{equation}
\label{eq:xi5} 
0<\xi<\theta \left(\frac{1}{f-1}  - \sqrt{3\varepsilon}\right).
\end{equation}
Note that the expression on the right side of the rightmost inequality above is positive by \eqref{eq:epsilon3}, and so it is possible to choose $\xi$ in this way. We use the assumption on the choice of partition to show that, for each $i$, every vertex in $A_i$ has few $G_2$-neighbours in $A_i\cap V_2$.

\begin{lem}
\label{lem:LittleBlueInsideClaim}
For $1\leq i\leq \chi(F_2)-1$, every $v\in A_i$ is adjacent in $G_2$ to fewer than $\theta |A_i|$ vertices of $A_i\cap V_2$.
\end{lem}

\begin{proof}
Let $v\in A_i$. Then $v$ is not bad by Lemma~\ref{lem:noV1}, so there must exist an index $i'$ such that the number of $G_2$-neighbours of $v$ in $A_{i'}\cap V_2$ is at most $\theta|A_{i'}|$. If $i'=i$, then we are done; so, we assume that $i'\neq i$. Since $V=V_0\sqcup V_2$ by Lemmas~\ref{lem:V1V2disjoint} and~\ref{lem:noV1}, the number of $G_2$-neighbours of $v$ in $A_{i'}$ overall is at most $\theta|A_{i'}| + |A_{i'}\cap V_0|$ which, by Lemma~\ref{lem:V0small}, is at most
\[\theta|A_{i'}| + \xi n < 2\theta |A_{i'}| \]
where the last step applies \eqref{eq:similarSizes} and \eqref{eq:xi5}. Since $i\neq i'$, the vertex $v$ must have at most $2\theta|A_{i'}|$ neighbours in $A_i$ as well; otherwise, moving $v$ from $A_i$ to $A_{i'}$ would decrease $\sum_{i=1}^{\chi(F_2)-1}e_2(A_i)$, contradicting our choice of partition. Thus,
\[d_2(v) \leq\sum_{j\notin\{ i,i'\}}|A_j| + 4\theta|A_{i'}|.\]
By \eqref{eq:similarSizes} and \eqref{eq:theta}, this is at most
\[(\chi(F_2)-3 + 4\theta)\left(\frac{1}{\chi(F_2)-1} +\sqrt{3\varepsilon}\right)n\leq\left( 1-  \frac{2}{\chi(F_2)-1} + 4\theta+ \chi(F_2)\sqrt{3\varepsilon}\right)n.\]
Using \eqref{eq:epsilon2}, we can bound this above by
\[\left( 1-  \frac{2}{\chi(F_2)-1} + 5\theta\right)n < \left( 1-  \frac{2}{\chi(F_2)-1} + \frac{5\theta\chi(F_2)}{\chi(F_2)-1}\right)(n-1) \]
\[= \left( 1-  \frac{1 + (1-5\theta\chi(F_2))}{\chi(F_2)-1}\right)(n-1) \]
By \eqref{eq:theta}, we have $\theta\leq\frac{1}{15f+5\chi(F_2)}$. Plugging this into the above expression yields an upper bound of 
\[\left(1-\frac{1+15\theta f}{\chi(F_2)-1}\right)(n-1)\]
contradicting Lemma~\ref{lem:noLargeRed} and completing the proof of the claim. 
\end{proof}

Next, let us show that every vertex $v\in A_i$ has many neighbours in $A_j$ for $j\neq i$. 

\begin{lem}
\label{lem:LittleBlueInside}
For $1\leq i\neq j\leq \chi(F_2)-1$, every $v\in A_i$ is adjacent in $G_2$ to at least $\left(1-33\theta f\right)|A_j|$ vertices of $A_j$.
\end{lem}

\begin{proof}
Let $v\in A_i$. Suppose that $v$ has fewer than $(1-33\theta f)|A_j|$ neighbours in $A_j$ in $G_2$. Let $S$ be the $G_2$-neighbourhood of $v$. Then
\[d_2(v)=\sum_{\ell=1}^{\chi(F_2)-1}|A_\ell\cap S| \leq|S\cap A_i\cap V_2| + |S\cap A_j|+|V_0| + \sum_{\ell\notin\{i,j\}}|A_\ell\cap S|\]
which, by Lemmas~\ref{lem:V0small} and \ref{lem:LittleBlueInsideClaim}, is at most
\[ \theta|A_i| + (1-33\theta f)|A_j| + \xi n+ n-|A_i|-|A_j|\]
\[= \left(1+\xi-\frac{(1-\theta)|A_i|+33\theta f|A_j|}{n}\right)n.\]
Using the lower bound on $|A_i|$ and $|A_j|$ in \eqref{eq:similarSizes} yields an upper bound of
\[\left(1+\xi-\left(1-\theta+33\theta f\right)\left(\frac{1}{\chi(F_2)-1} - \sqrt{3\varepsilon}\right)\right)n\]
\[\leq \left(1-\frac{1}{\chi(F_2)-1}+\xi + \sqrt{3\varepsilon} -32\theta f\left(\frac{1}{\chi(F_2)-1} - \sqrt{3\varepsilon}\right)\right)n.\]
By \eqref{eq:epsilon1}, \eqref{eq:epsilon2} and \eqref{eq:xi5}, this is less than
\[\left(1-\frac{1+15\theta f}{\chi(F_2)-1}\right)(n-1)\]
which contradicts Lemma~\ref{lem:noLargeRed} and completes the proof.
\end{proof}

Next, we prove that, in fact, there are no edges within $G_2[A_i]$ for any $1\leq i\leq \chi(F_2)-1$. The presence of critical edges in $F_1$ and $F_2$ is crucial in this step. After this, the proof of Theorem~\ref{th:hairy} will follow relatively easily. The following lemma is analogous to~\cite[Claim~3.12]{FoxWigderson23}. Assume that $t_0$ is large enough that the following holds for all $t\geq t_0$:
\begin{equation}
\label{eq:t0bound4}
\left(\frac{10}{9}\right)^{qt} > 4\chi\left(\frac{4\chi}{3}\right)^{ f}(t+  f)^2.
\end{equation}
Also, choose $\theta$ small enough so that
\begin{equation}
\label{eq:theta2}
0<\theta<\frac{1}{66f^3}.
\end{equation}
and
\begin{equation}
\label{eq:theta3}
(1-20\cdot\theta f)\geq (5/6)\cdot (1+15\theta f)^{1/q}.
\end{equation}
Note that such a $\theta$ exists because the limit as $\theta\to0$ of the left side is $1$ and the limit of the right size is $5/6$.

\begin{lem}
\label{lem:noBlueInside}
$e_2(A_i)=0$ for $1\leq i\leq \chi(F_2)-1$. 
\end{lem}

\begin{proof}
Suppose that the lemma is not true; without loss of generality, the set $A_1$ contains an edge of $G_2$. Let $u_0$ and $v_0$ be the endpoints of such an edge. Let $G_2'$ be the graph obtained from $G_2$ by deleting the edge $u_0v_0$ and let $G_1'=\overline{G_2'}$. We estimate the number of copies of $H_1$ in $G_1$ that are ``gained'' and the number of copies of $H_2$ in $G_2$ that are ``lost'' when replacing $(G_1,G_2)$ by $(G_1',G_2')$ with a goal of contradicting the choice of $G_1$. 

We begin by bounding from above the number of injective homomorphisms of $H_1$ to $G_1'$ which are not homomorphisms from $H_1$ to $G_1$. Any such homomorphism can be described as follows. First, we pick an edge $e=wz$ of $H_1$ and map its endpoints to $u_0v_0$ (in one of two possible ways). Now, imagine that we list the vertices of $H_1$ so that $w$ and $z$ are listed first (in this order), followed by the other vertices of the component of $H_1$ containing $w$ and $z$, and then the vertices of another (arbitrary) component, and so on, so that each vertex in the list is either the first vertex of its component or has a neighbour which comes before it in the list, which we refer to as its ``parent.'' Then, in a homomorphism, each vertex in the list after $w$ and $z$ must be mapped to a $G_1$-neighbour of its parent (if it has one). Thus, since $k(H_1)=k(F_1)$ and each vertex has at most $\left(\frac{1+15\theta f}{\chi(F_2)-1}\right)(n-1)$ neighbours in $G_1$ by Lemma~\ref{lem:noLargeRed}, the number of such mappings is at most
\[2e(H_1)n^{k(F_1)-1}\left(\frac{1+15\theta f}{\chi(F_2)-1}\right)^{v(H_1)-2-(k(F_1)-1)}n^{v(H_1)-2-(k(F_1)-1)}\]
\[=2e(H_1)(\chi(F_2)-1)\left(1+15\theta f\right)^{v(H_1)-k(F_1)-1}\left(\frac{1}{\chi(F_2)-1}\right)^{v(H_1)-k(F_1)}n^{v(H_1)-2}.\]
Thus, by \eqref{eq:samesame},
\[\rho_1\left(t_{\inj}(H_1,G_1')-t_{\inj}(H_1,G_1)\right)\]
\[\leq 2e(H_1)(\chi(F_2)-1)\left(1+15\theta f\right)^{v(H_1)-k(F_1)-1}n^{-2} + O(n^{-3}).\]
\[\leq 2e(H_1)(\chi(F_2)-1)\left(1+15\theta f\right)^{v(H_1)}n^{-2} + O(n^{-3}).\]

Next, let us bound from below the number of injective homomorphisms of $H_2$ to $G'_2$ which are not homomorphisms from $H_2$ to $G_2$. Let $e_0=w_0z_0$ be a critical edge of $F_2$, let $F_2'=F_2\setminus \{e_0\}$ and let $\psi:V(F_2')\to \{1,\dots, \chi(F_2)-1\}$ be a proper colouring of $F_2'$ such that $\psi(w_0)=\psi(z_0)=1$. Now, suppose that $\varphi$ is a function that maps $w_0$ to $u_0$ and $z_0$ to $v_0$ and then map every other vertex of $F_2'$ to $V$ randomly. The probability that every other vertex $u$ of $F_2'$ is mapped by $\varphi$ to $A_{\psi(u)}$ is
\[\prod_{u\in V(F_2')\setminus\{w_0,z_0\}} \left(\frac{|A_{\psi(u)}|}{n}\right) \geq \left(\frac{1}{\chi(F_2)-1} - \sqrt{3\varepsilon}\right)^{v(F_2')-2}\geq \left(\frac{1}{\chi}\right)^{ f}.\]
by \eqref{eq:similarSizes} and \eqref{eq:epsilon1}. Given this, by Lemma~\ref{lem:LittleBlueInside}, the probability that every edge of $F_2'$ maps to an edge of $G_2$ is, by a union bound, at least
\[1- 33\theta e(F_2') f\geq 1/2\]
where the inequality is by \eqref{eq:theta2}. Finally, given these two events, if each vertex of $V(H_2)\setminus V(F_2)$ is mapped randomly to $V$, an application of Lemma~\ref{lem:noLargeRed} combined with the above inequalities tells us that the probability that the final function is a homomorphism is at least
\[\frac{1}{2}\left(\frac{1}{\chi}\right)^{ f}\left(1-\frac{1+16\theta f}{\chi(F_2)-1}\right)^{v(H_2)-v(F_2)}=\frac{1}{2}\left(\frac{1}{\chi}\right)^{ f}\left(\frac{\chi(F_2)-2-16\theta f}{\chi(F_2)-1}\right)^{v(H_2)-v(F_2)}.\]
The last factor can be bounded as follows:
\[\left(\frac{\chi(F_2)-2-16\theta f}{\chi(F_2)-1}\right)^{v(H_2)-v(F_2)}\]
\[=\frac{\left(\chi(F_1)-1\right)^{v(H_2)-v(F_2)}}{(\chi(F_1)-1)^{v(H_2)-v(F_2)}}\left(\frac{\chi(F_2)-2-16\theta f}{\chi(F_2)-1}\right)^{v(H_2)-v(F_2)}\]
\[=\left(\frac{1}{\chi(F_1)-1}\right)^{v(H_2)-v(F_2)}\left(\frac{(\chi(F_1)-1)(\chi(F_2)-2-16\theta f)}{\chi(F_2)-1}\right)^{v(H_2)-v(F_2)}\]
\[\geq\left(\frac{1}{\chi(F_1)-1}\right)^{v(H_2)-k(F_2)}\left(\frac{(\chi(F_1)-1)(\chi(F_2)-2-16\theta f)}{\chi(F_2)-1}\right)^{v(H_2)-v(F_2)}\] where in the inequality we used that $v(F_2) \geq k(F_2)$.
Now, since $\chi(F_1)\geq 3$ and $\chi(F_1)+\chi(F_2)\geq 7$, we have 
\[\frac{(\chi(F_1)-1)(\chi(F_2)-2-16\theta f)}{\chi(F_2)-1} \geq \frac{4}{3}(1-20\cdot \theta f).\]
Putting this all together and applying \eqref{eq:samesame}, we get that
\[\rho_2(t_{\inj}(H_2,G_2)-t_{\inj}(H_2,G_2'))\]
\[\geq \frac{1}{2}\left(\frac{1}{\chi}\right)^{ f}\left(\frac{4}{3}(1-20\cdot\theta f)\right)^{v(H_2)-v(F_2)}n^{-2} + O(n^{-3}) \geq \frac{1}{2}\left(\frac{3}{4\chi}\right)^{ f}\left(\frac{4}{3}(1-20\cdot\theta f)\right)^{v(H_2)}n^{-2} + O(n^{-3}) \]
which, by \eqref{eq:relSizes}, is at least
\[\frac{1}{2}\left(\frac{3}{4\chi}\right)^{ f}\left(\frac{4}{3}(1-20\cdot\theta f)\right)^{qv(H_1)}n^{-2} + O(n^{-3}).\]
Now, by \eqref{eq:theta3}, this is at least
\[\frac{1}{2}\left(\frac{3}{4\chi}\right)^{ f}\left(\frac{4}{3}\cdot \frac{5}{6}(1+15\theta f)^{1/q}\right)^{qv(H_1)}n^{-2} + O(n^{-3})\]
\[=\frac{1}{2}\left(\frac{3}{4\chi}\right)^{ f}\left(\frac{10}{9}\right)^{qv(H_1)}(1+15\theta f)^{v(H_1)}n^{-2} + O(n^{-3}).\]

Combining the upper bound that we have proven on $\rho_1\left(t_{\inj}(H_1,G_1')-t_{\inj}(H_1,G_1)\right)$ and the lower bound on $\rho_2(t_{\inj}(H_2,G_2)-t_{\inj}(H_2,G_2'))$, we get that 
\[n^2(m(H_1,H_2;G_1)-m(H_1,H_2;G_1'))\]
\[\geq \frac{1}{2}\left(\frac{3}{4\chi}\right)^{ f}\left(\frac{10}{9}\right)^{qv(H_1)}(1+15\theta f)^{v(H_1)} -2e(H_1)(\chi(F_2)-1)\left(1+15\theta f\right)^{v(H_1)} - O(n^{-1})\]
which is positive for large $n$ by \eqref{eq:t0bound4}. This contradicts the definition of $G_1$ and completes the proof. 
\end{proof}

Finally, we present the proof of Theorem~\ref{th:hairy}.

\begin{proof}[Proof of Theorem~\ref{th:hairy}] 
Given $F_1, F_2$ and $q$ satisfying the hypotheses of the theorem, we select our parameters in the following order, subject to the given conditions:
\begin{itemize}
\item choose $\theta$ to satisfy \eqref{eq:theta}, \eqref{eq:theta2} and \eqref{eq:theta3},
\item choose $\varepsilon$ so that \eqref{eq:epsilon1}, \eqref{eq:epsilon2} and \eqref{eq:epsilon3} hold,
\item choose $\delta$ as in \eqref{eq:delta},
\item choose $\beta$ as in \eqref{eq:beta}, 
\item choose $\xi$ so that \eqref{eq:xi1}, \eqref{eq:xi2}, \eqref{eq:xi3}, \eqref{eq:xi4} and \eqref{eq:xi5} all hold,
\item choose $\gamma$ as in \eqref{eq:gamma},
\item choose $\tau$ to satisfy \eqref{eq:taua}, \eqref{eq:taub} and \eqref{eq:tauc},
\item choose $t_0$ large enough so that \eqref{eq:t0bound1}, \eqref{eq:t0bound2}, \eqref{eq:t0bound3} and \eqref{eq:t0bound4} all hold.
\end{itemize}
Let $t_1,t_2\geq t_0$ and let $H_1$ be a $t_1$-hairy $F_1$ and $H_2$ be a $t_2$-hairy $F_2$. We may assume that $H_1$ and $H_2$ have no singleton components. Let $n_0$ be large with respect to $H_1$ and $H_2$ and the parameters chosen in the previous paragraph and let $G_1$ be a graph on $n$ vertices minimizing $m(H_1,H_2;G_1)$  and $G_2=\overline{G_1}$. Without loss of generality, $|V_1|\leq |V_2|$. As a result of our parameter choices, all of the statements in Sections~\ref{sec:hairy} and~\ref{sec:hairier} hold. In particular, Lemma~\ref{lem:noBlueInside} implies that there is a partition $A_1,\dots, A_{\chi(F_2)-1}$ of $V=V(G_1)$ such that $G_2$ contains no edge with endpoints in $A_i$ for $1\leq i\leq \chi(F_2)-1$, and \eqref{eq:similarSizes} guarantees that all of the sets of the partition have approximately the same size, $\frac{n}{\chi(F_2)-1}$.

We assert that $G_1$ has no edges between $A_i$ and $A_j$ for $i \neq j$. To prove this, suppose that such an edge exists in $G_1$. If we move this edge from $G_1$ to $G_2$, it would destroy at least one injective homomorphism from $H_1$ to $G_1$ (since $|A_i|>v(H_1)$ for large $n$ and $H_1$ has at least $t_1\geq1$ vertices of degree one). At the same time, this would not create any injective homomorphism from $H_2$ to $G_2$, since $G_2$ is still $(\chi(F_2)-1)$-partite after adding such an edge to it. This contradicts our choice of $G_1$. Therefore, $G_2$ is a complete $(\chi(F_2)-1)$-partite graph. In particular, $t(H_2,G_2)=0$ and $G_1$ is a disjoint union of $\chi(F_2)-1$ cliques.

Finally, we show that the cardinalities of any two sets $A_i$ and $A_j$ differ by at most one. 
Each homomorphism from $H_1$ to $G_1$ gives rise to a partition of $V(H_1)$ into at most $\chi(F_2)-1$ classes such that each partition class is a union of components of $H_1$ and all vertices of each class are mapped to the same component of $G_1$. We think of these partitions as being ``unlabelled'' in the sense that they contain information about which components of $H_1$ are mapped to the same component of $G_1$ but not about which component of $G_1$ they are mapped to. Given such a partition $\mathcal{P}=\{P_1,\dots,P_{\chi(F_2)-1}\}$ (where we allow some of the sets $P_j$ to be empty), we show that the number of injective homomorphisms of $H_1$ to $G_1$ giving rise to the partition $\mathcal{P}$ is minimized when the cardinalities any two of the sets $A_i$ and $A_j$ differ by at most one. For each $1\leq i\leq \chi(F_2)-1$, let us count the number of choices for the mapping of vertices in $P_i$ given the mapping of the vertices in $\bigcup_{j=1}^{i-1}P_j$. Let $T_i$ be the set of indices $t$ such that there does not exist $1\leq j\leq i-1$ such that the vertices of $P_j$ are mapped to $A_t$. Then the number of choices for the mapping of $P_i$ given that of $P_j$ for all $j<i$ is 
\begin{equation}\label{eq:ithclass}\sum_{t\in T_i}\frac{|A_t|!}{(|A_t|-|P_i|)!}\end{equation}
For an integer $c\geq2$, define $f_c:\mathbb{R}\to \mathbb{R}$ by $f_c(z)=z(z-1)\dots (z-c+1)$. Then $f_c$ has $c-1$ distinct (integer) roots in the interval $[0,c-1]$. The derivative $f_c'(z)$ is a polynomial of degree $c-1$ with $c-1$ real roots which interlace the roots of $f_c$; in particular, its roots are also contained in the interval $[0,c-1]$. By similar logic, the roots of the second derivative $f_c''$ are in $[0,c-1]$ as well. From this, we see that $f_c''$ is positive on $[c,\infty)$, and so $f_c$ is strictly convex on this set. Thus, by Jensen's Inequality, for any $i$ such that $|P_i|\geq 2$, the sum in \eqref{eq:ithclass} is uniquely minimized when the cardinalities of the sets $A_t$ for $t\in T_i$ are as similar as possible. Thus, the number of injective homomorphisms from $H_1$ to $G_1$ is minimized by taking $G_2$ to be a $(\chi(F_2)-1)$-partite Tur\'an graph. 
\end{proof}

\section{Beating the Tur\'an Colouring}
\label{sec:neg}

In this section, we show that if $(H_1,H_2)$ is a bonbon pair, then $e(H_1)$ cannot be ``excessively large;'' see Theorem~\ref{th:imbalance} below. This result will then be used to derive Proposition~\ref{prop:imbalance}. We will use the following result of~\cite{FoxHeManners19} which was previously known as Tomescu's Graph Colouring Conjecture~\cite{Tomescu71}.

\begin{thm}[Fox, He and Manners~\cite{FoxHeManners19}]
\label{th:Tomescu}
For $m\neq 3$, every connected $m$-chromatic graph on $n$ vertices has at most $m!(m-1)^{n-m}$ proper $m$-colourings\footnote{For clarity, given the previous sections, the colourings here are for vertices, not edges.}.
\end{thm}

The case $m=2$ of the above theorem is trivial, as every connected bipartite graph has precisely two proper $2$-colourings. It is also necessary to exclude the case $m=3$, as an odd cycle of length $k\geq5$ has more than $3!2^{k-3}$ proper $3$-colourings. Knox and Mohar~\cite{KnoxMohar19,KnoxMohar20} established the cases $m=4$ and $m=5$ before the full conjecture was proven by Fox, He and Manners~\cite{FoxHeManners19}. Note that every hairy $K_m$ on $n$ vertices has exactly $m!(m-1)^{n-m}$ proper $m$-colourings and so Theorem~\ref{th:Tomescu} is tight. We will use the following corollary of Theorem~\ref{th:Tomescu}. Given a graph $H$, say that a vertex colouring $f:V(H)\to [\chi(H)-1]$ of $H$ is \emph{nearly proper} if there is a unique edge of $H$ whose endpoints are monochromatic.

\begin{cor}
\label{cor:nearlyProper}
If $H$ is a graph such that $\chi(H)\neq 4$, then the number of nearly proper colourings of $H$ is at most
\[\crit(H)\cdot(\chi(H)-2)!\cdot(\chi(H)-2)^{v(H)-\chi(H)-k(H)+1}\cdot(\chi(H)-1)^{k(H)}.\]
\end{cor}

\begin{proof}
Given an edge $e$ of $H$, let $H/e$ be the graph obtained by \emph{contracting $e$}; i.e. by identifying the two endpoints of $e$ and removing any multi-edges that arise. Let $z_e$ be the vertex formed by contracting the edge $e$. The number of nearly proper colourings of $H$ is equal to the number of ways to select 
\begin{itemize}
    \item a critical edge $e$ of $H$,
    \item a proper $(\chi(H)-1)$-colouring of the component of $H/e$ containing $z_e$ and
    \item a proper $(\chi(H)-1)$-colouring of the components of $H/e$ that do not contain $z_e$. 
\end{itemize}
The number of choices in the first step is clearly $\crit(H)$. 

Assuming that a critical edge $e$ has been chosen, let $p$ denote the number of vertices in the component of $H/e$ containing $z_e$. Note that the chromatic number of this component is exactly $\chi(H)-1$ which, since $\chi(H)\neq 4$, is not equal to three. So, by Theorem~\ref{th:Tomescu}, the number of choices in the second step is at most $(\chi(H)-1)!(\chi(H)-2)^{p-(\chi(H)-1)}$.

In the last step, for each component that does not contain $z_e$, there are at most $\chi(H)-1$ choices for the colour of an arbitrary ``root'' vertex of this component and then at most $\chi(H)-2$ choices for each subsequent vertex. Since $H/e$ has $v(H)-1$ vertices, the number of vertices in the components of $H/e$ that do not contain $z_e$ is $v(H)-1 -p$. Thus, the number of choices in the last step is at most
\[(\chi(H)-1)^{k(H)-1}(\chi(H)-2)^{v(H)-1-p-(k(H)-1)}.\]
Putting this all together, we get that the number of nearly proper colourings of $H$ is at most
\[\crit(H)\cdot (\chi(H)-1)!\cdot (\chi(H)-2)^{v(H)-\chi(H)-k(H)+1}\cdot (\chi(H)-1)^{k(H)-1}\]
as desired.
\end{proof}

We also need the following simple bound on the number of nearly proper colourings in the case that $\chi(H)=4$. The proof is analogous to that of the previous corollary, except that, instead of Theorem~\ref{th:Tomescu}, we use the (trivial) fact that every connected $3$-chromatic graph on $n$ vertices has at most $3\cdot 2^{n-1}$ proper $3$-colourings. 

\begin{lem}
\label{lem:trivialchi4}
If $H$ is a $4$-chromatic graph, then the number of nearly proper colourings of $H$ is at most
\[\crit(H)\cdot 3^{k(H)}\cdot 2^{v(H)-k(H)-1}.\]
\end{lem}

Next, we use Corollary~\ref{cor:nearlyProper} and Lemma~\ref{lem:trivialchi4} to prove the following result which restricts the number of edges in a graph contained in a bonbon pair. In fact, it applies to a slightly more general class of graphs. Say that $(H_1,H_2)$ is a \emph{multiplicity good} if
\[(\chi(H_2)-1)^{v(H_1)-k(H_1)}t(H_1,G) + (\chi(H_1)-1)^{v(H_2)-k(H_2)}t(H_2,\overline {G})\geq 1-o(1)\] for all graphs $G$.
Clearly, every bonbon pair is multiplicity good. Say that a graph $H$ is a \emph{multiplicity good} if the pair $(H,H)$ is.

\begin{thm}
\label{th:imbalance}
Let $H_1$ and $H_2$ be graphs such that, if $\chi(H_2)\neq 4$, then
\[e(H_1)>\frac{\crit(H_2)\cdot (\chi(H_2)-2)!\cdot(\chi(H_2)-2)^{v(H_2)-\chi(H_2)-k(H_2)+1}\cdot(\chi(H_1)-1)^{v(H_2)-k(H_2)}}{(\chi(H_2)-1)^{v(H_2)-k(H_2)}}\]
and, otherwise, 
\[e(H_1)>\frac{\crit(H_2)\cdot 2^{v(H_2)-k(H_2)-1}(\chi(H_1)-1)^{v(H_2)-k(H_2)}}{3^{v(H_2)-k(H_2)}}.\]
Then $(H_1,H_2)$ is not multiplicity good. 
\end{thm}

\begin{proof}
Suppose that $H_1$ and $H_2$ are graphs satisfying the hypotheses of the theorem. Let $\varepsilon>0$ be very small and, for each $n\geq \chi(H_2)-1$, let $G_{n,\varepsilon}$ be a graph on $n$ vertices obtained from the complement of the Tur\'an graph with $n$ vertices and $\chi(H_2)-1$ parts by jettisoning each edge of this graph with probability $\varepsilon$ independently of all other such edges. Define
\[f_1(\varepsilon):=\lim_{n\to\infty}t(H_1,G_{n,\varepsilon}),\qquad f_2(\varepsilon):=\lim_{n\to\infty}t(H_2,\overline{G_{n,\varepsilon}})\]
and note that both of these limits exist with probability $1$. Moreover, with probability one,
\[(\chi(H_2)-1)^{v(H_1)-k(H_1)}\cdot f_1(\varepsilon) =(1-\varepsilon)^{e(H_1)}= 1-e(H_1)\varepsilon+O(\varepsilon^2)\]
where the asymptotics here (and throughout the proof) are as $\varepsilon\to0$. 

Let $K$ be the number of nearly proper colourings of $H_2$. Then, with probability one,
\begin{equation}\label{eq:f2thing}(\chi(H_1)-1)^{v(H_2)-k(H_2)}\cdot f_2(\varepsilon) = (\chi(H_1)-1)^{v(H_2)-k(H_2)}\cdot\varepsilon K\left(\frac{1}{\chi(H_2)-1}\right)^{v(H_2)} + O(\varepsilon^2).\end{equation}
At this point, we divide the proof into cases.

\begin{casee}
$\chi(H_2)\neq 4$.
\end{casee}

By Corollary~\ref{cor:nearlyProper}, the linear term of \eqref{eq:f2thing} (with respect to $\varepsilon$) is at most
\[(\chi(H_1)-1)^{v(H_2)-k(H_2)}\cdot\varepsilon \cdot \crit(H_2)\cdot(\chi(H_2)-2)!\cdot(\chi(H_2)-2)^{v(H_2)-\chi(H_2)-k(H_2)+1}\left(\frac{1}{\chi(H_2)-1}\right)^{v(H_2)-k(H_2)}\]
which is equal to 
\[\frac{\varepsilon\cdot \crit(H_2)\cdot (\chi(H_2)-2)!\cdot(\chi(H_2)-2)^{v(H_2)-\chi(H_2)-k(H_2)+1}\cdot(\chi(H_1)-1)^{v(H_2)-k(H_2)}}{(\chi(H_2)-1)^{v(H_2)-k(H_2)}}.\]
Therefore, the lower bound on $e(H_1)$ assumed at the beginning of the proof implies that the linear term with respect to $\varepsilon$ in $(\chi(H_2)-1)^{v(H_1)-k(H_1)}\cdot f_1(\varepsilon) + (\chi(H_1)-1)^{v(H_2)-k(H_2)}\cdot f_2(\varepsilon)$ has a negative coefficient. So, for $\varepsilon$ sufficiently small, we have that 
\[(\chi(H_2)-1)^{v(H_1)-k(H_1)}t(H_1,G_{n,\varepsilon}) + (\chi(H_1)-1)^{v(H_2)-k(H_2)}t(H_2,\overline {G_{n,\varepsilon}})=1-\Omega(\varepsilon)\] 
as $n\to\infty$ which implies that $(H_1,H_2)$ is not multiplicity good.

\begin{casee}
$\chi(H_2)=4$. 
\end{casee}

In this case, by Lemma~\ref{lem:trivialchi4}, the linear term of \eqref{eq:f2thing} with respect to $\varepsilon$ is at most
\[(\chi(H_1)-1)^{v(H_2)-k(H_2)}\varepsilon\cdot\crit(H_2)\cdot 2^{v(H_2)-k(H_2)-1}\left(\frac{1}{3}\right)^{v(H_2)-k(H_2)}\]
\[=\frac{\varepsilon\cdot \crit(H_2)\cdot 2^{v(H_2)-k(H_2)-1}(\chi(H_1)-1)^{v(H_2)-k(H_2)}}{3^{v(H_2)-k(H_2)}}.\]
Thus, analogous to the previous case, by taking $\varepsilon$ sufficiently close to zero, we get a certificate that $(H_1,H_2)$ is not multiplicity good. 
\end{proof}

\begin{proof}[Proof of Proposition~\ref{prop:imbalance}]
The proposition follows immediately from Theorem~\ref{th:imbalance}.
\end{proof}

\section{Conclusion}
\label{sec:concl}

We conclude by stating some open problems. A result of Goodman~\cite{Goodman59} implies that $c_1(K_3)=1/4$ and so $K_3$ is multiplicity good. However, for odd $n$, the quantity $\hom_{\inj}(K_3,G)+\hom_{\inj}(K_3,\overline{G})$ is minimized among all $n$-vertex graphs by every $n$-vertex graph $G$ which is $((n-1)/2)$-regular; therefore, $K_3$ is multiplicity good but not a bonbon. We are currently unaware of any non-$3$-colourable graph which is multiplicity good but not a bonbon, which leads us to the following question. 

\begin{ques}
Is it true that every non-$3$-colourable multiplicity good graph is a bonbon?
\end{ques}

It would also be interesting to explore off-diagonal variants of the above question, such as the following.

\begin{ques}
Suppose that $(H_1,H_2)$ is multiplicity good such that $H_1$ and $H_2$ are non-bipartite and $\chi(H_1)+\chi(H_2)\geq7$. Does it follow that $(H_1,H_2)$ is a bonbon pair?
\end{ques}

Currently, all of the known examples of bonbons contain vertices of degree one. It is unclear whether a bonbon of minimum degree at least two can exist. The analogous question for non-3-colourable multiplicity good graphs is also intriguing (the case of chromatic number three is settled, since $K_3$ is multiplicity good).

\begin{ques}
Does there exist a bonbon $H$ such that $\delta(H)\geq2$? 
\end{ques}

\begin{ques}
Does there exist a non-3-colourable multiplicity good graph $H$ such that $\delta(H)\geq2$? 
\end{ques}

\begin{ack}
The authors would like to thank Elena Moss for valuable discussions on themes related to those covered in this paper.
\end{ack}

\end{document}